\documentclass{svmult}
\usepackage{enumerate}
\usepackage{graphics}
\usepackage{amsmath}
\usepackage{amssymb}
\usepackage{amsfonts}

\newcommand{\eps}{\varepsilon}

\newcommand{\s}{\qquad}

\newcommand{\be}{\begin{equation}}
\newcommand{\ba}{\begin{array}}
\newcommand{\ee}{\end{equation}}
\newcommand{\ea}{\end{array}}
\newcommand{\oln}{\overline}
\newcommand{\tx}{\text}
\newcommand{\ds}{\displaystyle}
\def\eop{{\ \vrule height 3pt width 3pt depth 0pt}}
\begin{document}\author{
{\sc V.~Franklin}$^1$\;\;{\sc M.~Paramasivam}$^1$\;\;\sc
{S.~Valarmathi}$^1$\;\;{and}\;\;{\sc J.J.H.~Miller}$^2$}
\title*{{\bf Second order parameter-uniform convergence for a finite difference method for a singularly
perturbed linear parabolic system}} \institute{$^1${Department of
Mathematics, Bishop Heber College, Tiruchirappalli-620017, Tamil
Nadu, India.\\
franklinvicto@gmail.com, sivambhcedu@gmail.com, valarmathi07@gmail.com.}\\
$^2${Institute for Numerical Computation and Analysis, Dublin 2,
Ireland. jm@incaireland.org.}}
\titlerunning{Numerical solution of a parabolic reaction-diffusion system}
\maketitle
\begin{abstract}
A singularly perturbed linear system of second order partial differential equations of parabolic reaction-diffusion type with given initial and boundary conditions
is considered. The leading term of each equation is multiplied by a small positive parameter. These singular perturbation parameters are assumed to be distinct. The
components of the solution exhibit overlapping layers. Shishkin piecewise-uniform meshes are introduced, which are used in conjunction with a classical finite
difference discretisation, to construct a numerical method for solving this problem. It is proved that the numerical approximations obtained with this method are
first order convergent in time and essentially second order convergent in the space variable uniformly with respect to all of the parameters.
\end{abstract}
\section{Introduction}

The following parabolic initial-boundary value problem is considered
for a singularly perturbed linear system of second order
differential equations
\begin{equation}\label{BVP}
\frac{\partial\vec{u}}{\partial t}
-E\frac{\partial^2\vec{u}}{\partial
x^2}+A\vec{u}=\vec{f},\;\;\mathrm{on}\;\;
\Omega,\;\;\vec{u}\;\;\mathrm{given\;\; on}\;\;\Gamma,
\end{equation}
where $\Omega=\{(x,t): 0 < x < 1, 0 < t \leq T\}$,\;\;
$\overline{\Omega}=\Omega \cup \Gamma,\;\; \Gamma = \Gamma_L \cup
\Gamma_B \cup \Gamma_R$ with $ \vec u(0,t)=\vec \phi_L(t)\;
\mathrm{on}\;\Gamma_L=\{(0,t): 0 \leq t \leq T\},\;\; \vec
u(x,0)=\vec \phi_B(x)\;\mathrm{on}\;\Gamma_B=\{(x,0): 0 \leq x \leq
1\},\;\; \vec u(1,t)=\vec \phi_R(t)\; \mathrm{on}\;\Gamma_R=\{(1,t):
0 \leq t \leq T\}.$ Here, for all $(x,t) \in \overline{\Omega},$
$\;\vec{u}(x,t)$ and $\vec{f}(x,t)$ are column
$n-\mathrm{vectors},\;E\;$ and $\;A(x,t)\;$ are $\;n\times n\;$
matrices, $\;E = \mathrm{diag}(\vec{\eps}),\;\vec\eps =
(\eps_1,\;\cdots,\;\eps_n)\;$ with $\;0\;<\;\eps_i\;\le\;1\;$ for
all $\;i=1,\ldots,n$. The $\eps_i$ are assumed to be distinct and,
for convenience, to have the ordering\[\eps_1\;<\;\cdots\;<\;\eps_n.\] Cases with some of the parameters coincident are not considered here.\\
\noindent The problem can also be written in the operator form
\[\vec L\vec u\; = \;\vec{f}\;\;\mathrm{on}\;\;
\Omega,\;\;\vec{u}\;\;\mathrm{given\;\; on}\;\;\Gamma,\] where the
operator $\;\vec{L}\;$ is defined by
\[\vec{L}\;=\;\frac{\partial}{\partial t}-E\frac{\partial^2}{\partial x^2} + A.\]
For all $(x,t) \in \overline{\Omega}$ it is assumed
that the components $a_{ij}(x,t)$ of $A(x,t)$
satisfy the inequalities \\
\begin{eqnarray}\label{a1} a_{ii}(x,t) > \displaystyle{\sum_{^{j\neq
i}_{j=1}}^{n}}|a_{ij}(x,t)| \; \; \rm{for}\;\; 1 \le i \le n, \;\; \rm{and} \;\;a_{ij}(x,t) \le 0 \;\; \rm{for} \; \; i \neq j\end{eqnarray} and, for some $\alpha$,
\begin{eqnarray}\label{a2} 0 <\alpha <
\displaystyle{\min_{^{(x,t)\in \overline\Omega}_{1 \leq i \leq
n}}}(\sum_{j=1}^n a_{ij}(x,t)).
\end{eqnarray}
It is also assumed, without loss of generality, that
\begin{eqnarray}\label{a3}
\max_{1 \leq i \leq n} \sqrt{\eps_i} \leq \frac{\sqrt{\alpha}}{6}.
\end{eqnarray}
The reduced problem corresponding to \eqref{BVP} is defined by
\begin{equation}
\frac{\partial\vec{u}_0}{\partial t}
+A\vec{u}_0=\vec{f},\;\;\mathrm{on}\;\;
\Omega,\;\;\vec{u}_0=\vec{u}\;\;\mathrm{on}\;\;\Gamma_B.
\end{equation}
The norms $\parallel \vec{V} \parallel =\max_{1 \leq k \leq n}|V_k|$
for any n-vector $\vec{V}$, $\parallel y
\parallel_D =\sup\{|y(x,t)|: (x,t)\in D\}$ for any
scalar-valued function $y$ and domain $D$, and $\parallel \vec{y}
\parallel=\max_{1 \leq k \leq n}\parallel y_{k}
\parallel$ for any vector-valued function $\vec{y}$ are introduced.
When $D=\overline{\Omega}$ or $\Omega$ the subscript $D$ is usually
dropped. Throughout the paper $C$ denotes a generic positive
constant, which is independent of $x, t$ and of all singular
perturbation and discretization parameters. Furthermore,
inequalities between vectors
are understood in the componentwise sense. Whenever necessary the required smoothness of the problem data is assumed.\\

\noindent For a general introduction to parameter-uniform numerical methods for singular perturbation problems, see \cite{MORS}, \cite{RST} and  \cite{FHMORS}. The
piecewise-uniform Shishkin meshes $\Omega^{M,N}$ in the present paper have the
 elegant property that they reduce to uniform meshes when
the parameters are not small. The problem posed in the present paper
is also considered in \cite{GLOR}, where parameter uniform
convergence is proved, which is first order in time and essentially
first order in space. The meshes used there do not have the above
typical property of Shishkin meshes. The main result of the present
paper is well known in the scalar case, when $n=1$. It is
established in \cite{L} for the case $n=2$. The proof below of first
order convergence in the time variable and essentially second order
convergence in the space variable, for general $n$,  draws heavily
on the analogous result in \cite{PVM} for a reaction-diffusion
system.

The plan of the paper is as follows. In the next two sections both standard and novel bounds on the smooth and singular components of the exact solution are
obtained. The sharp estimates for the singular component in Lemma \ref{lsingular}
 are proved by mathematical induction, while interesting orderings of the points $x_{i,j}$ are
established in Lemma \ref{layers}. In Section 4
 piecewise-uniform Shishkin meshes are introduced, in Section 5 the
discrete problem is defined and the discrete maximum principle, the discrete stability properties and a comparison result are established. In Section 6 an
expression for the local truncation error is derived and standard estimates are stated. In Section 7 parameter-uniform estimates for the local truncation error of
the smooth and singular components are obtained in a sequence of theorems. The section culminates with the statement and proof of the essentially second order
parameter-uniform error
estimate.\\

\section{Standard analytical results}
The operator $\vec L$ satisfies the following maximum principle
\begin{lemma}\label{max} Let $A(x,t)$ satisfy (\ref{a1}) and (\ref{a2}). Let
$\;\vec{\psi}\;$ be any function in the domain of $\;\vec L\;$ such
that $\vec{\psi}\ge \vec{0}$ on $\Gamma$. Then $\;\vec
L\vec{\psi}(x)\ge\vec{0}\;$ on $\Omega$ implies that
$\;\vec{\psi}(x)\ge\vec{0}\;$ on $\overline{\Omega}$.
\end{lemma}
\begin{proof}Let $i^*, x^*, t^*$ be such that $\psi_{i^*}(x^{*}, t^*)=\min_{i}\min_{\overline{\Omega}}\psi_i(x,t)$
and assume that the lemma is false. Then $\psi_{i^*}(x^{*},t^*)<0$ .
From the hypotheses we have $(x^*,t^*) \not\in\;\Gamma$ and
$\frac{\partial^2 \psi_{i^*}}{\partial x^2}(x^*,t^*)\geq 0$. Thus
\begin{equation*}(\vec{L}\vec \psi(x^*,t^*))_{i^*}= \frac{\partial \psi_{i^*}}{\partial t}(x^*,t^*)
-\eps_{i^*}\frac{\partial^2\psi_{i^*}}{\partial
x^2}(x^*,t^*)+\sum_{j=1}^n a_{i^*,j}(x^{*},t^*)\psi_j(x^*,t^*)<0,
\end{equation*} which contradicts the assumption and proves the
result for $\vec{L}$.\eop \end{proof}

Let $\tilde{A}(x,t)$ be any principal sub-matrix of $A(x,t)$ and $\vec{\tilde{L}}$ the corresponding operator. To see that any $\vec{\tilde{L}}$ satisfies the same
maximum principle as $\vec{L}$, it suffices to observe that the elements of $\tilde{A}(x,t)$ satisfy \emph{a fortiori} the same inequalities as those of $A(x,t)$.
\begin{lemma}\label{stab} Let $A(x,t)$ satisfy (\ref{a1}) and (\ref{a2}). If $\vec{\psi}$ is any function in the domain of $\;\vec L,\;$
 then, for each $i, \; 1 \leq i \leq n $ and $(x,t)\in \overline{\Omega}$, \[|\psi_i(x,t)| \le\;
\max\ds\left\{\parallel\vec{\psi}\parallel_{\Gamma}, \dfrac{1}{\alpha}\parallel \vec L\vec \psi\parallel\right\}.\]
\end{lemma}
\begin{proof}Define the two functions \[\vec \theta^\pm(x,t)\;=\;
\max\ds\left\{\parallel\vec{\psi}\parallel_{\Gamma},\;\dfrac{1}{\alpha}\parallel\vec{L}\vec \psi\parallel\right\}\vec e\;\pm\;\vec \psi(x,t)\] where $\;\vec
e\;=\;(1,\;\ldots,\;1)^T\;$ is the unit column vector. Using the properties of $\;A\;$ it is not hard to verify that $\vec \theta^\pm\;\ge\;\vec 0$ on $\Gamma$ and
$\;\vec L\vec \theta^\pm\;\ge\;\vec 0$ on $\Omega.$ It follows from Lemma \ref{max} that $\;\vec \theta^\pm\;\ge\;\vec 0\;$ on $\overline{\Omega}$ as required.\eop
\end{proof}
A standard estimate of the exact solution and its derivatives is
contained in the following lemma.
\begin{lemma}\label{lexact} Let $A(x,t)$ satisfy (\ref{a1}) and (\ref{a2})
and let $\vec u$ be the exact solution of (\ref{BVP}). Then, for all $(x,t)\in \oln \Omega$ and each $i=1,\dots,n$,
\begin{equation*}\label{e1}
\begin{array}{lcl}
|\frac{\partial^l u_i}{\partial t^l}(x,t)| &\leq& C(||\vec{u}||_\Gamma+\sum_{q=0}^{l}||\frac{\partial^q\vec{f}}{\partial t^q}||),\;\; l= 0,1,2\\\\
|\frac{\partial^l u_i}{\partial x^l}(x,t)| &\leq& C\eps_i^{\frac{-l}{2}}(||\vec{u}||_\Gamma+||\vec{f}||+||\frac{\partial\vec{f}}{\partial t}||), \;\; l=1,2\\\\
|\frac{\partial^l u_i}{\partial x^l}(x,t)| &\leq&
C\eps^{-1}_i \eps^{\frac{-(l-2)}{2}}_1 (||\vec{u}||_\Gamma+||\vec{f}||+
||\frac{\partial\vec{f}}{\partial t}||+||\frac{\partial^2 \vec f}{\partial t^2}||+
\eps^{\frac{l-2}{2}}_1||\frac{\partial^{l-2} \vec f}{\partial x^{l-2}}||), \;\; l=3,4\\\\
|\frac{\partial^l u_i}{\partial^{l-1}x \partial t}(x,t)| &\leq& C\eps_i^{\frac{1-l}{2}}(||\vec{u}||_\Gamma+||\vec{f}||+||\frac{\partial\vec{f}}{\partial
t}||+||\frac{\partial^2 \vec f}{\partial t^2}||), \;\; l=2,3.
\end{array}
\end{equation*}
\end{lemma}
\begin{proof}
The bound on $\vec{u}$ is an immediate consequence of Lemma \ref{stab}.
Differentiating (\ref{BVP}) partially with respect to '$t$' once and twice, and applying Lemma \ref{stab} the bounds $ \frac{\partial \vec u}{\partial t}$, and $
\frac{\partial^2 \vec u}{\partial t^2} $ are obtained.
To bound $\frac{\partial u_i}{\partial x}$, for all $i$ and any $(x,t)$, consider an interval $I=(a, a+\sqrt \eps_i)$ such that $x \in I$.\\
Then for some $y\in I$ and $t \in (0,T]$
\[\s\frac{\partial u_i}{\partial x}(y,t)=\frac{u_i(a+\sqrt{\eps_i},t)-u_i(a,t)}{\sqrt{\eps_i}}\]

\begin{equation} \label{e2}
|\frac{\partial u_i}{\partial x}(y,t)|\leq C {\eps_i}^{\frac{-1}{2}}||\vec u||.\s\s\;\;\;
\end{equation}
Then for any $ x \in I $
\[\frac{\partial u_i}{\partial x}(x,t) =\frac{\partial u_i}{\partial x}(y,t)+\int_y^x\frac{\partial ^2 u_i(s,t)}{\partial x^2}ds \]
\[\frac{\partial u_i}{\partial x}(x,t) =\frac{\partial u_i}{\partial x}(y,t)+\eps_i^{-1}\int_y^x\left(\frac{\partial u_i(s,t)}{\partial t} - f_i(s,t)+\sum_{j=1}^{n} a_{ij}(s,t)u_j(s,t)\right)ds \]
\[|\frac{\partial u_i}{\partial x}(x,t)| \leq |\frac{\partial u_i}{\partial x}(y,t)|+C\eps_i^{-1}\int_y^x(||\vec u||_\Gamma + ||\vec f||+||\frac{\partial \vec f}{\partial t}||)ds.\]
Using (\ref{e2}) in the above equation
\[|\frac{\partial u_i}{\partial x}(x,t)|\leq C {\eps_i}^{\frac{-1}{2}}(||\vec u||_\Gamma+||\vec f||+||\frac{\partial \vec f}{\partial t}||).\]

Rearranging the terms in (\ref{BVP}), it is easy to get
\[|\frac{\partial^2 u_i}{\partial x^2}|\leq C \eps_i^{-1}(||\vec u||_{\Gamma}+||\vec f||+||\frac{\partial \vec f}{\partial t}||).\]
Analogous steps are used to get the rest of the estimates. \s\s\s\eop \end{proof}
The Shishkin decomposition of the exact solution $\;\vec{u}\;$ of
(\ref{BVP}) is $\;\vec{u}=\vec{v}+\vec{w}\;$ where the smooth
component $\;\vec v\;$ is the solution of
\begin{equation}\label{smoothcomp}\;\vec L\vec v = \vec f\; \; \mathrm{in} \; \Omega,
 \;\vec
v = \vec u_0\; \; \mathrm{on} \; \Gamma
\end{equation} and the singular component $\;\vec w\;$
is the solution of \begin{equation}\label{singularcomp} \vec L\vec
w\;=\;\vec 0 \; \mathrm{in} \; \Omega, \;\vec w = \vec u-\vec v\; \;
\mathrm{on} \; \Gamma.
\end{equation} For convenience the left and right boundary layers of
$\vec w$ are separated using the further decomposition $\vec w
=\vec{w}^L+\vec{w}^R$ where $\vec{L}\vec{w}^L=\vec{0}$\;on
$\Omega,\; \vec{w}^L=\vec w$ \; on $\Gamma_L,\; \vec{w}^L=\vec{0}$
on $\Gamma_B \cup \Gamma_R$ and $\vec{L}\vec{w}^R= \vec{0}$ \; on
$\Omega,\; \vec{w}^R=\vec{w}$\;on $\Gamma_R$, $\vec{w}^R=\vec{0}$
on $\Gamma_L \cup \Gamma_B$.\\
Bounds on the smooth component and its derivatives are contained in
\begin{lemma}\label{lsmooth1} Let $A(x,t)$ satisfy (\ref{a1}) and (\ref{a2}).
Then the smooth component $\vec v$ and its derivatives satisfy, for all $(x,t) \in \overline \Omega$ and each $i=1,\dots,n$,
\begin{equation*} \label{e5}
\begin{array}{lcl}
|\frac{\partial^l v_i}{\partial t^l}(x,t)| &\leq& C \; \mathrm{for}\; l=0,1,2\\\\
|\frac{\partial^l v_i}{\partial x^l}(x,t)| &\leq& C(1+\eps_i^{1-\frac{l}{2}})\; \mathrm{for}\; l=0,1,2,3,4\\\\
|\frac{\partial^l v_i}{\partial x^{l-1} \partial t}(x,t)| &\leq& C\;\; \text{for}\;\; l=2,3.
\end{array}
\end{equation*}
\end{lemma}
\begin{proof}
The bound on $\vec{v}$ is an immediate consequence of the defining
equations for $\vec{v}$ and Lemma (\ref{stab}). Differentiating the
equation \eqref{smoothcomp} twice partially with respect to $x$ and
applying Lemma \ref{stab} for $\frac{\partial^2 v_i}{\partial x^2}$,
we get
\begin{equation}\label{e5a}
|\frac{\partial^2v_i}{\partial x^2}|\leq C(1+||\frac{\partial \vec v}{\partial x}||).
\end{equation}
Let
\begin{equation}\label{e6}
\frac{\partial v_{i^*}}{\partial x}(x^*,t*) = ||\frac{\partial
\vec v}{\partial x}|| \s \text{for some}\; i=i^* , \; x=x^*,\;
t=t^*.
\end{equation}
Using Taylor expansion, it follows that, for some $ y \in [0,1-x^*]$ and some $\eta \in (x^*, x^*+y)$
\begin{equation}\label{e7} v_{i^*}(x^*+y,t^*)=v_{i^*}(x^*,t^*)+y\frac{\partial v_{i^*}}{\partial x}(x^*,t^*)+\frac{y^2}{2} \frac{\partial^2 v_{i^*}}{\partial x^2}(\eta,t^*).
\end{equation}
Rearranging (\ref{e7}) yields
\[ \frac{\partial v_{i^*}}{\partial x}(x^*,t^*)= \frac{v_{i^*}(x^*+y,t^*)-v_{i^*}(x^*,t^*)}{y}-\frac{y}{2} \frac{\partial^2v_{i^*}}{\partial x^2}(\eta,t^*) \]
\begin{equation} \label{e8}
|\frac{\partial v_{i^*}}{\partial x}(x^*,t^*)| \leq
\frac{2}{y}||\vec v||+\frac{y}{2}||\frac{\partial^2 \vec
v}{\partial x^2}||.
\end{equation}
Using (\ref{e6}) and (\ref{e8})  in (\ref{e5a}),
\[|\frac{\partial^2v_i}{\partial x^2}|\leq C(1+\frac{2}{y}||\vec v||+\frac{y}{2}||\frac{\partial ^2 \vec v}{\partial x^2}||).\]
This leads to
\[(1-\frac{Cy}{2}) ||\frac{\partial^2 \vec v}{\partial x^2}||\leq C(1+\frac{2}{y}||\vec v||)\]
or\\
\begin{equation}\label{e9}
||\frac{\partial^2 \vec v}{ \partial x^2}|| \leq C.
\end{equation}
Using (\ref{e9}) in (\ref{e8}) yields
\[||\frac{\partial \vec v}{\partial x}|| \leq C. \]
Repeating the above steps with $\frac{\partial v_i}{\partial t}$,
it is easy to get the required bounds on the mixed derivatives.
The bounds on $\dfrac{\partial^3 \vec v}{\partial
x^3},\;\;\dfrac{\partial^4 \vec v}{\partial x^4} $ are derived by
a similar argument.\s\s\s\s\s\s\s\s\s\s\s\s\s\s\eop
\end{proof}
\section{Improved estimates}
The layer functions $B^{L}_{i}, \; B^{R}_{i}, \; B_{i}, \; i=1,\;
\dots , \; n,\;$, associated with the solution $\;\vec u$, are
defined on $[0,1]$ by
\[B^{L}_{i}(x) = e^{-x\sqrt{\alpha/\eps_i}},\;B^{R}_{i}(x) =
B^{L}_{i}(1-x),\;B_{i}(x) = B^{L}_{i}(x)+B^{R}_{i}(x).\] The
following elementary properties of these layer functions, for all $1
\leq i < j \leq n$ and $0 \leq x < y \leq
1$, should be noted:\\
(a)\;$B^{L}_i(x)\; <\; B^{L}_j(x),\;\;B^{L}_i(x)\;
>\; B^{L}_i(y), \;\;0\;<\;B^{L}_i(x)\;\leq\;1$.\\
(b)\;$B^{R}_i(x)\; <\; B^{R}_j(x),\;\;B^{R}_i(x)\; <\;
B^{R}_i(y), \;\;0\;<\;B^{R}_i(x)\;\leq\;1$.\\
(c)\;$B_{i}(x)$ is monotone decreasing (increasing) for
increasing $x \in [0,\frac{1}{2}] ([\frac{1}{2},1])$.\\
(d)\;$B_{i}(x) \leq 2B_{i}^L(x)$ for $x \in [0,\frac{1}{2}]$,
\;$B_{i}(x) \leq 2B_{i}^R(x)$ for $x \in [\frac{1}{2},1]$.
\begin{definition}
For $B_i^L$, $B_j^L$, each $i,j, \;\;1 \leq i \neq j \leq n$ and
each $s, s>0$, the point $x^{(s)}_{i,j}$ is defined by
\begin{equation}\label{x1}\frac{B^L_i(x^{(s)}_{i,j})}{\varepsilon^s _i}=
\frac{B^L_j(x^{(s)}_{i,j})}{\varepsilon^s _j}. \end{equation}
\end{definition}
It is remarked that
\begin{equation}\label{x2}\frac{B^R_i(1-x^{(s)}_{i,j})}{\varepsilon^s _i}=
\frac{B^R_j(1-x^{(s)}_{i,j})}{\varepsilon^s _j}. \end{equation} In the next lemma the existence and uniqueness of the points $x^{(s)}_{i,j}$  are shown. Various
properties are also established.
\begin{lemma}\label{layers} For all $i,j$, such that $1 \leq i < j \leq
n$ and $0<s \leq 3/2$,  the points $x_{i,j}$ exist, are uniquely
defined and satisfy the following inequalities
\begin{equation}\label{x3}
\frac{B^L_{i}(x)}{\eps^s _i} > \frac{B^L_{j}(x)}{\eps^s _j},\;\; x \in [0,x^{(s)}_{i,j}),\;\; \frac{B^L_{i}(x)}{\eps^s _i} <
\frac{B^L_{j}(x)}{\eps^s _j}, \; x \in (x^{(s)}_{i,j}, 1].\end{equation}\\
Moreover
\begin{equation}\label{x4}x^{(s)}_{i,j}< x^{(s)}_{i+1,j}, \; \mathrm{if} \;\; i+1<j \;\;
\mathrm{and} \;\; x^{(s)}_{i,j}<
x^{(s)}_{i,j+1}, \;\; \mathrm{if} \;\; i<j. \end{equation} \\
Also
\begin{equation}\label{x5}
x^{(s)}_{i,j}< 2s\sqrt{\frac{\eps_j}{\alpha}}\;\; and \;\;
x^{(s)}_{i,j} \in
(0,\frac{1}{2})\;\; \mathrm{if} \;\; i<j. \end{equation}\\
Analogous results hold for the $B^R_i$, $B^R_j$ and the points $1-x^{(s)}_{i,j}.$\\
\end{lemma}
\begin{proof} Existence, uniqueness and (\ref{x3}) follow
from the observation that the ratio of the two sides of (\ref{x1}),
namely
\[\frac{B^L_{i}(x)}{\eps^s _i}\frac{\eps^s _j}{B^L_{j}(x)}=
\frac{\eps^s _j}{\eps^s _i} \exp{(-\sqrt{\alpha}
x(\frac{1}{\sqrt{\eps_i}}-\frac{1}{\sqrt{\eps_j}}))},\] is
monotonically decreasing from the value $\frac{\eps^s_j}{\eps^s_i}
>1$ as $x$ increases
from $0$.\\
The point $x^{(s)}_{i,j}$ is the unique point $x$ at which this
ratio has the value $1.$ Rearranging (\ref{x1}), and using the
inequality $\ln x <x-1$ for all $x>1$, gives
\begin{equation}\label{x5}x^{(s)}_{i,j} = 2s\ds\left[
\frac{\ln(\frac{1}{\sqrt{\eps_i}})-
\ln(\frac{1}{\sqrt{\eps_j}})}{\sqrt{\alpha}(\frac{1}{\sqrt{\eps_i}}-
\frac{1}{\sqrt{\eps_j}})}\right]=
\frac{2s\;\ln(\frac{\sqrt{\eps_j}}{\sqrt{\eps_i}})}{\sqrt{\alpha}(\frac{1}{\sqrt{\eps_i}}-
\frac{1}{\sqrt{\eps_j}})}<
2s\sqrt{\frac{\eps_j}{\alpha}},\end{equation} which
is the first part of (\ref{x5}). The second part follows immediately from this and (\ref{a3}).\\
To prove (\ref{x4}), writing $\sqrt{\eps_k} = \exp(-p_k)$, for some
$p_k
> 0$ and all $k$, it follows that
\[x^{(s)}_{i,j}=\frac{2s(p_i -p_j)}{\sqrt{\alpha}(\exp{p_i} -\exp{p_j})}.\] The
inequality $x^{(s)}_{i,j}< x^{(s)}_{i+1,j}$ is equivalent to
\[\frac{p_i -p_j}{\exp{p_i} -\exp{p_j}}<\frac{p_{i+1} -p_j}{\exp{p_{i+1}} -\exp{p_j}}, \]
which can be written in the form
\[(p_{i+1}-p_j)\exp(p_i-p_j)+(p_{i}-p_{i+1})-(p_{i}-p_j)\exp(p_{i+1}-p_j)>0. \]
With $a=p_i-p_j$ and $b=p_{i+1}-p_j$ it is not hard to see that
$a>b>0$ and $a-b=p_i-p_{i+1}$. Moreover, the previous inequality is
then equivalent to
\[\frac{\exp{a}-1}{a}>\frac{\exp{b}-1}{b}, \] which is true because $a>b$ and proves
the first part of (\ref{x4}). The second part is proved by a similar argument.\\ The analogous results for the $B^R_i$, $B^R_j$ and the points $1-x^{(s)}_{i,j}$ are
proved by a similar argument.\eop
\end{proof}
%
In the following lemma sharper estimates of the smooth component are
presented.
\begin{lemma}\label{lsmooth2}
Let $\;A(x,t)\;$ satisfy (\ref{a1}) and (\ref{a2}). Then the smooth component $\;\vec v\;$ of the solution $\;\vec u\;$ of \eqref{BVP} satisfies for all
$\;i=1,\cdots,n$ and all $\;(x,t)\;\in\oln\Omega$
\[ |\frac{\partial^l v_i}{\partial x^l}(x,t)| \leq  C\left (1+\sum_{q=i}^{n}\frac{B_q(x)}{\eps_q ^{\frac{l}{2}-1}}\right) \; \mathrm{for}\; l=0,1,2,3.\]
\end{lemma}

\begin{proof}Define two barrier functions
\[\vec\psi^\pm(x,t)\;=\;C[1+B_n(x)]\vec e\;\pm\;\frac{\partial^l\vec v}{\partial x^l}(x,t),\;\;l=0,1,2\;\;\;\text{and}\;\;\;(x,t)\in\overline\Omega.\]

We find that, for a proper choice of C,

\[\psi_i^\pm(0,t)\;=\; C\;\pm\;\frac{\partial^l u_{0,i}}{\partial x^l}(0,t)=C \geq 0\]
\[\psi_i^\pm(1,t)\;=\; C\;\pm\;\frac{\partial^l u_{0,i}}{\partial x^l}(1,t)=C \geq 0\]
\[\psi_i^\pm(x,0)\;=\; C[1+B_n(x)]\;\pm\;\frac{\partial^l \phi_{B,i}(x)}{\partial x^l}= C[1+B_n(x)]\;\pm\;  C \geq\ 0\]
as $\phi_b(x) \in C^{(2)}(\Gamma_b)$
and $(\vec L \vec\psi^{\pm})_i(x,t)\geq 0$.\\
Using Lemma \ref{max}, we conclude that
\begin{equation}\label{e15}
|\frac{\partial^l v_i}{\partial x^l}(x,t)| \leq C[1+B_n(x)]\;\; \text{for}\;\;\; l=0,1,2.
\end{equation}
Consider the equation
\begin{equation}\label{e16}(\vec L(\frac{\partial^2 \vec v}{\partial x^2}))_i\;=\; \frac{\partial^2 f_i}{\partial x^2}-2 \frac{\partial \sum_{j=1}^{n} a_{ij}}{\partial x} \frac{\partial v_j}{\partial x}-\frac{\partial^2 \sum_{j=1}^{n}a_{ij}}{\partial x^2}v_j
\end{equation}

with\\
\begin{equation}\label{e17}
\frac{\partial^2 v_i}{\partial x^2}(0,t)\;=\;0,\frac{\partial^2 v_i}{\partial x^2}(1,t)\;=\;0, \frac{\partial^2 v_i}{\partial x^2}(x,0)\;=\;\frac{\partial^2
\phi_{b,i}(x)}{\partial x^2}.
\end{equation}
For convenience, let $\vec p$ denote $\frac{\partial^2\vec v}{\partial x^2}$. Then
\begin{equation}\label{e18}
\vec \L\vec p =\vec g \;\;\text{with}\;\; \vec p(0,t) =\vec 0,\;\; \vec p(1,t) =\vec 0,\;\; \vec p(x,0) =\vec s
\end{equation}
where \[g_i = \frac{\partial^2 f_i}{\partial x^2} -2\frac{\partial \sum_{j=1}^{n} a_{ij}}{\partial x}\frac{\partial v_j}{\partial x}-\sum_{j=1}^{n}\frac{\partial^2
a_{ij}}{\partial x^2}v_j \;\;\text{and}\;\; s_i =  \frac{\partial^2 \phi_{b,i}}{\partial x^2}(x).\] Let $\vec q$ and $\vec r$ be the smooth and singular components
of $\vec p$ given by
\begin{equation}\label{e19}
\vec L \vec q = \vec g \;\;\text{with}\;\; \vec q(0,t) =\vec p_0(0,t),\;\; \vec q(1,t) =\vec p_0(1,t),\;\; \vec q(x,0) =\vec p(x,0)
\end{equation}
where $\vec p_0$ is the solution of the reduced problem
\[\frac{\partial \vec p_0}{\partial t} + A \vec p_0=\vec g \;\;\text{with}\;\; \vec p_0(x,0) = \vec p(x,0)= \vec s. \]\\
Now,
\begin{equation}\label{e20}
\vec L \vec r = \vec 0,\;\;\text{with}\;\; \vec r(0,t) = - \vec q (0,t),\; \vec r(1,t) = -\vec q(1,t),\; \vec r(x,0) = \vec 0.
\end{equation}
Using Lemma \ref{lsmooth1} and Lemma \ref{lsingular}, we have for $i=1,\dots,n$ and $(x,t) \in \overline \Omega$
\[ |\frac{\partial q_i }{\partial x}(x,t)| \leq C\]
and \[|\frac{\partial r_i }{\partial x}(x,t)| \leq C[\frac{B_i(x)}{\sqrt {\eps_i}}+\dots+\frac{B_n(x)}{\sqrt {\eps_n}}].\]\\
Hence, for $(x,t) \in \overline \Omega\;$and $i=1,\dots,n$,
\begin{equation}\label{e21}
|\frac{\partial^3 v_i }{\partial x^3}| =|\frac{\partial p_i }{\partial x}| \leq C[1+\frac{B_i(x)}{\sqrt {\eps_i}}+\dots+\frac{B_n(x)}{\sqrt {\eps_n}}].
\end{equation}
From (\ref{e15}) and (\ref{e21}), we find that for $l=0,1,2,3$ and $(x,t)\in \overline \Omega$
\[|\frac{\partial^l v_i }{\partial x^l}| \leq C[1+\eps_i^{1-\frac{l}{2}}B_i(x)+\dots+
\eps_i^{1-\frac{l}{2}}B_n(x)].\s\s\s\eop\]  \end{proof}
{\bf Remark :} It is interesting to note that the above estimate
reduces to the estimate of the smooth component of the solution of
the scalar problem given in \cite{MORS} when $\;n=1.\;$\\
Bounds on the singular components $\vec{w}^L,\; \vec{w}^R$ of $\vec{u}$ and their derivatives are contained in\\
\begin{lemma}\label{lsingular} Let $A(x,t)$ satisfy (\ref{a1}) and
(\ref{a2}). Then there exists a constant $C,$ such that, for each $(x,t) \in \oln \Omega$ and $i=1,\; \dots , \; n$,
\begin{equation*} \begin{array} {l}
|\frac{\partial^l w^L_i}{\partial t^l}(x,t)| \le C B^L_{n}(x),\;\mathrm{for} \;\; l=0,1,2. \\\\
|\frac{\partial^l w^L_i}{\partial x^l}(x,t)| \le C\sum_{q=i}^n \frac{B^L_{q}(x)}{\eps_q^{\frac{l}{2}}},\;\mathrm{for}\;\; l=1,2.\\\\
|\frac{\partial^3 w^L_i}{\partial x^3}(x,t)| \le C\sum_{q=1}^n \frac{B^L_{q}(x)}{\eps_q^{\frac{3}{2}}}.\\\\
|\frac{\partial^4 w^L_i}{\partial x^4}(x,t)| \le C\frac{1}{\eps_i} \sum_{q=1}^n \frac{B^L_{q}(x)}{\eps_q}.

\end{array} \end{equation*}
Analogous results hold for $w^R_i$ and its derivatives.
\end{lemma}
\begin{proof}To obtain the bound of $\vec{w}^L$, define the functions
${\psi_i}^\pm(x,t)\;=\;Ce^{\alpha t}B_n^L(x)\;\pm\;w_i^L(x,t)$, for each $i=1,\dots,n$. It is clear that
${\psi_i}^\pm(0,t)$, ${\psi_i}^\pm(x,0)$, ${\psi_i}^\pm(1,t)$ and $(\vec L\vec\psi^\pm)_i(x,t)$ are non-negative. By Lemma 1, ${\psi_i}^\pm(x,t)\geq0$. It follows that $|w_i^L|\leq Ce^{\alpha t}B_n^L(x)$ \\
or
\begin{equation}\label{e22}
|w_i^L|\leq CB_n^L(x).
\end{equation}
To obtain the bound for $\frac{\partial w_i^L}{\partial t}$, define
the two functions
${\theta_i}^\pm(x,t)\;=\;CB_n^L(x)\;\pm\;\frac{\partial
w_i^L}{\partial t}(x,t)$ for each $i=1,\dots,n$. Differentiating the
homogeneous equation satisfied by $w_i^L$, partially with respect to
$t$, and rearranging yields
\[\frac{\partial^2 w_i^L}{\partial t^2} - \eps_i \frac{\partial^3 w^L_i}{\partial x^2 \partial t} + \sum_{j=1}^{n} a_{ij}\frac{\partial w_j^L}{\partial t} = \frac{-\partial\sum_{j=1}^{n}a_{ij}}{\partial t}w_j^L,\]
and we get
\[|L\frac{\partial w_i^L}{\partial t}|\leq C{B_n}^L(x)\]
\[|\frac{\partial w_i^L}{\partial t}(0,t)|\leq |\frac{\partial {u_i}}{\partial t}(0,t)|+|\frac{\partial {v_i}}{\partial t}(0,t)|=|\frac{\partial {\phi}_{L,i}(t)}{\partial t}|\leq C\]
\[|\frac{\partial w_i^L}{\partial t}(1,t)|\leq |\frac{\partial {u_i}}{\partial t}(1,t)|+|\frac{\partial {v_i}}{\partial t}(1,t)|=|\frac{\partial {\phi}_{R,i}(t)}{\partial t}|\leq C\]
\[|\frac{\partial w_i^L}{\partial t}(x,0)|\leq |\frac{\partial {\phi}_{B,i}(x)}{\partial t}|=0.\]
By Lemma \ref{stab}, it follows that
 \begin{equation}\label{e23}
|\frac{\partial w_i^L}{\partial t}|\leq C{B_n}^L(x).
\end{equation}
Now the bound for $\frac{\partial^2 w^L_i}{\partial x \partial t}$ is obtained by using Lemma (\ref{lexact}) and Lemma (\ref{lsmooth1})
\[ |\frac{\partial^2 w^L_i}{\partial x \partial t}|\leq |\frac{\partial^2 u_i}{\partial x \partial t}|+|\frac{\partial^2 v_i}{\partial x \partial t}|\]
\[ |\frac{\partial^2 w^L_i}{\partial x \partial t}| \leq C {\eps_i}^{\frac{-1}{2}}({||\vec u||_\Gamma}+||\vec f||+|| \frac{\partial \vec f}{\partial t}||+||\frac{\partial^2 \vec f}{\partial t^2}||).\]
Similarly,
\[ |\frac{\partial^3 w^L_i}{\partial x^2 \partial t}| \leq C \eps_i^{-1}({||\vec u||_\Gamma}+||\vec f||+|| \frac{\partial \vec f}{\partial t}||+||\frac{\partial^2 \vec f}{\partial t^2}||). \]
The bounds on $\frac{\partial^l w_i^L}{\partial x^l}, l=1,2,3,4$ and
$i=1,\dots,n$ are derived by the method of induction on $n$. It is
assumed that the bounds $\frac{\partial w_i^L}{\partial x} ,
\frac{\partial^2 w_i^L}{\partial x^2},\frac{\partial^3
w_i^L}{\partial x^3}$ and $\frac{\partial^4 w_i^L}{\partial x^4}$
hold for all systems up to $n-1$. Define
$\vec{\tilde{w}}^L=(w_1^L,\dots,w_{n-1}^L)$, then
$\vec{\tilde{w}}^L$
 satisfies the system  \[\frac{\partial \vec{\tilde{w}}^L}{\partial t}-\tilde{E}\frac{\partial^2 \vec{\tilde{w}}^L}{\partial x^2}+\tilde A \vec{\tilde{w}}^L = \vec g, \]
with
\[\vec{\tilde{w}}^L(0,t) = \vec{\tilde{u}}(0,t)-\vec{\tilde{u}}_0(0,t), \vec{\tilde{w}}^L(1,t) =\vec{\tilde{0}},\]

\[\vec{\tilde{w}}^L(x,0)=\vec{\tilde{u}}(x,0)-\vec{\tilde{u}}_0(x,0) = \vec{\tilde{\phi}}_B(x)-\vec{\tilde{\phi}}_B(x)=\vec{\tilde{0}}.\]
Here, $\tilde{E}$ and $\tilde{A}$ are the matrices obtained by
deleting the last row and column from $E,A$ respectively, the
components of $\vec g$ are $g_i = -a_{i,n}w_n^L$ for $1\leq i \leq
n-1$ and $\vec{\tilde{u}}_0$ is the solution of the reduced problem.
Now decompose $\vec{\tilde{w}}^L$ into smooth and singular
components to get $\vec{\tilde{w}}^L = \vec{q} +
\vec{r},\;\frac{\partial \vec{\tilde{w}}^L}{\partial x} =
\frac{\partial \vec q}{\partial x} + \frac{\partial \vec r}{\partial
x} $. By induction, the bounds on the derivatives of ${ \vec{\tilde
w}}^L$ hold. That is for $i=1,\dots,n-1$

\begin{equation} \label{e24}
\ds \left.
\begin{array}{lcl}
|\dfrac{\partial w_i^L}{\partial x}| &\leq& C \sum_{q=i}^{n-1} \eps_q^{\frac{-1}{2}}B_q^L(x) \\\\
|\dfrac{\partial^2 w_i^L}{\partial x^2}| &\leq& C \sum_{q=i}^{n-1} \eps_q^{-1}{B_q}^L(x) \\\\
|\dfrac{\partial^3 w_i^L}{\partial x^3}| &\leq& C \sum_{q=1}^{n-1} \eps_q^{\frac{-3}{2}}B_q^L(x) \\\\
|\eps_i \dfrac{\partial^4 w_i^L}{\partial x^4}| &\leq& C \sum_{q=1}^{n-1} \eps_q^{-1}B_q^L(x)
\end{array}
\right \}
\end {equation}
Rearranging the $n^{th}$ equation of the system satisfied by $w_n^L$
yields
\[\eps_n \frac{\partial^2 w_n^L}{\partial x^2 }=\frac{\partial w_n^L}{\partial t} +  \sum_{j=1}^{n} a_{nj}{w_j^L}.\]
Using (\ref{e22}) and (\ref{e23}) gives
\begin{equation} \label{e25}
 |\frac{\partial^2 w_n^L}{\partial x^2 }| \leq C \eps_n^{-1} B_n^L(x).
\end{equation}
Applying the mean value theorem to $w_n^L$ at some $y$, $a < y< a+\sqrt \eps_n$

\[ \frac{\partial w_n^L}{\partial x}(y,t)= \frac{w_n^L(a+\sqrt \eps_n,t)-w_n^L(a,t)}{\sqrt \eps_n} \]
Using (\ref{e22}) gives
\[|\frac{\partial w_n^L}{\partial x}(y,t)|\leq \frac{C}{\sqrt \eps_n}(B_n^L(a+\sqrt \eps_n)+B_n^L(a)). \]
So
\begin{equation} \label{e26}
|\frac{\partial w_n^L}{\partial x}(y,t)|\leq \frac{C}{\sqrt \eps_n}B_n^L(a).
\end{equation}
Again
\begin{equation} \label{e27}
 \frac{\partial w_n^L}{\partial x}(x,t) =  \frac{\partial w_n^L}{\partial x}(y,t) + (y-x) \frac{\partial^2 w_n^L}{\partial x^2}(\eta,t),\;\;\; y<\eta<x.
\end{equation}
Using (\ref{e25}) and (\ref{e26}) in (\ref{e27}) yields
\[\begin{array}{lcl}
|\frac{\partial w_n^L}{\partial x}(x,t)| &\leq& C[\eps_n^{\frac{-1}{2}}B_n^L(a)+ \eps_n^{\frac{-1}{2}}B_n^L(\eta)] \\
 &\leq& C \eps_n^{\frac{-1}{2}}B_n^L(a)\\
 &=& C \eps_n^{\frac{-1}{2}}B_n^L(x)\frac{B_n^L(a)}{B_n^L(x)}\\
 &=& C \eps_n^{\frac{-1}{2}}B_n^L(x) e^{(x-a)\sqrt \alpha / \sqrt \eps_n}\\
 &=& C \eps_n^{\frac{-1}{2}}B_n^L(x) e^{\sqrt \eps_n \sqrt \alpha / \sqrt \eps_n}.\\
\end{array}\]
Therefore
\begin{equation} \label{e28}
|\frac{\partial w_n^L}{\partial x}(x,t)|\leq C \eps_n^{\frac{-1}{2}}B_n^L(x).\;\;\s\s\s
\end{equation}
Now, differentiating the equation satisfied by $w_n^L$ partially with respect to $x$, and rearranging, gives
\[\eps_n  \frac{\partial^3 w_n^L}{\partial x^3} = \frac{\partial^2 w_n^L}{\partial x \partial t} + \sum_{q=1}^{n-1}a_{nq}\frac{\partial w_q^L}{\partial x } + a_{nn}\frac{\partial w_n^L}{\partial x} + \sum_{q=1}^{n} \frac{\partial a_{nq}}{\partial x} w_q^L. \]
The bounds on $w_n^L$ and (\ref{e24}) then give
\[|\frac{\partial^3 w_n^L}{\partial x^3}| \leq C \sum_{q=1}^{n} {\eps_q}^{\frac{-3}{2}}B_q^L(x). \]
Similarly
\[|\eps_n \frac{\partial^4 w_n^L}{\partial x^4}| \leq C \sum_{q=1}^{n} \eps_q^{-1}B_q^L(x).\]

Using the bounds on $w_n^L, \frac{\partial w_n^L}{\partial x},\frac{\partial^2 w_n^L}{\partial x^2},\frac{\partial^3 w_n^L}{\partial x^3}$ and $\frac{\partial^4
w_n^L}{\partial x^4}$, it is seen that $\vec g,\;\frac{\partial \vec g}{\partial x},$ $\frac{\partial^2 \vec g}{\partial x^2},$ $\frac{\partial^3 \vec g}{\partial
x^3},$ $\frac{\partial^4 \vec g}{\partial x^4}$ are bounded by $ CB_n^L(x),$ $C\frac{B_n^L(x)}{\sqrt \eps_n},$ $C\frac{B_n^L(x)}{\eps_n},$ $\sum_{q=1}^{n}
\frac{B_q^L(x)}{\eps_q^{\frac{3}{2}}},$ $C \eps_n^{-1} \sum_{q=1}^{n} \frac{B_q^L(x)}{\eps_q} $ respectively.
Introducing the functions ${\vec \psi}^\pm(x,t)\;=\;CB_n^L(x)\vec e\;\pm\;\vec q(x,t)$, it is easy to see that ${\vec \psi}^\pm(0,t)=C\vec e\;\pm\;\vec q(0,t) \geq
\vec0$, ${\vec \psi}^\pm(1,t)=CB_n^L(1)\vec e\;\pm\;\vec 0 \geq \vec 0$,  ${\vec \psi}^\pm(x,0)=CB_n^L(x)\vec e\;\pm\;\vec 0 \geq \vec 0$ and
\[(\vec L {\vec \psi}^\pm)_i(x,t)=C(-\eps_i \frac{\alpha}{\eps_n}+\sum_{j=1}^{n}a_{ij})B_n^L(x)\;\pm\;CB_n^L(x) \geq 0.\]
Applying Lemma 1, it follows that $||\vec q(x,t)|| \leq C B_n^L(x)$. Defining barrier functions ${\vec \theta}^\pm (x,t) = C \eps_n^{\frac{-1}{2}}B_n^L(x) \vec
e\;\pm\; \frac{\partial \vec q}{\partial x} $ and using Lemma 3 for the problem satisfied by $\vec q$, the bound required for $\frac{\partial \vec q}{\partial x} $
and $\frac{\partial^2 \vec q}{\partial x^2} $ is obtained. By induction, the following bounds for \vec r are obtained for $i=1,\dots,n-1,$
\[ |\frac{\partial r_i}{\partial x} |\leq [\frac{B_i^L(x)}{{\sqrt \eps_i}}+\dots+\frac{B_{n-1}^L(x)}{{\sqrt \eps_{n-1}}}   ],\]
\[ |\frac{\partial^2 r_i}{\partial x^2} |\leq C[\frac{B_i^L(x)}{{ \eps_i}}+\dots+\frac{B_{n-1}^L(x)}{{ \eps_{n-1}}}],\]
\[ |\frac{\partial^3 r_i}{\partial x^3} |\leq C[\frac{B_1^L(x)}{{ \eps_1^{\frac{3}{2}}}}+\dots+\frac{B_{n-1}^L(x)}{{ \eps_{n-1}^\frac{3}{2}}}],\]
\[ |\eps_i\frac{\partial^4 r_i}{\partial x^4} |\leq C[\frac{B_1^L(x)}{{ \eps_1}}+\dots+\frac{B_{n-1}^L(x)}{{ \eps_{n-1}}}].\]
Combining the bounds for the derivatives of $q_i$ and $r_i$ it follows that, for $i= 1,2,\dots,n$
\[ \begin{array}{rcl}
|\frac{\partial^l w^L_i}{\partial x^l}|&\leq& |\frac{\partial^l q_i}{\partial x^l}|+|\frac{\partial^l r_i}{\partial x^l }|\\\\
|\frac{\partial^l w^L_i}{\partial x^l}|&\leq& C \sum_{q=i}^{n} \frac{B_q^L(x)}{{ \eps_q^{\frac{l}{2}}}}\;\;\text{for}\;\; l= 1,2\\\\
|\frac{\partial^3 w^L_i}{\partial x^3}|&\leq& C \sum_{q=1}^{n} \frac{B_q^L(x)}{{ \eps_q^{\frac{3}{2}}}}\\\\
\text{and}\;\;\;\; |\eps_i\frac{\partial^4 w^L_i}{\partial x^4}|&\leq& C \sum_{q=1}^{n} \frac{B_q^L(x)}{\eps_q}. \end{array}\] Recalling the bounds on the
derivatives of $w_n^L$ completes the proof of the lemma for the system of $n$ equations.\\ A similar proof of the analogous results for the right boundary layer
functions holds.\eop\end{proof}

\section{The Shishkin mesh}
 A piecewise
uniform Shishkin mesh with $M \times N$ mesh-intervals is now
constructed. Let $\Omega^M_t=\{t_k \}_{k=1}^{M},\;\;\Omega^N_x=\{x_j
\}_{j=1}^{N-1},\;\;\overline{\Omega}^M_t=\{t_k
\}_{k=0}^{M},\;\;\overline{\Omega}^N_x=\{x_j
\}_{j=0}^{N},\;\;\Omega^{M,N}=\Omega^M_t \times \Omega^N_x,\;\;
 \overline{\Omega}^{M,N}=\overline{\Omega}^M_t \times \overline{\Omega}^N_x \;\;\mathrm{and}\;\;\Gamma^{M,N}=\Gamma \cap \overline{\Omega}^{M,N}.$
 The mesh $\overline{\Omega}^M_t$ is chosen to be a uniform mesh with $M$
mesh-intervals on $[0,T]$. The mesh $\overline{\Omega}^N_x$ is a
piecewise-uniform mesh on $[0,1]$ obtained by dividing $[0,1]$ into
$2n+1$ mesh-intervals as follows
\[[0,\sigma_1]\cup\dots\cup(\sigma_{n-1},\sigma_n]\cup(\sigma_n,1-\sigma_n]\cup(1-\sigma_n,1-\sigma_{n-1}]\cup\dots\cup(1-\sigma_1,1].\]
The $n$ parameters $\sigma_r$, which determine the points separating the uniform meshes, are defined by
\begin{equation}\label{tau1}\sigma_{n}=
\min\displaystyle\left\{\frac{1}{4},2\sqrt{\frac{\eps_n}{\alpha}}\ln N\right\}\end{equation} and for $\;r=1,\;\dots \; ,n-1$
\begin{equation}\label{tau2}\sigma_{r}=\min\displaystyle\left\{\frac{\sigma_{r+1}}{2},2\sqrt{\frac{\eps_r}{\alpha}}\ln
N\right\}.\end{equation} Clearly \[
0\;<\;\sigma_1\;<\;\dots\;<\;\sigma_n\;\le\;\frac{1}{4}, \qquad
\frac{3}{4}\leq 1-\sigma_n < \; \dots \;< 1-\sigma_1 <1.\] Then, on
the sub-interval $\;(\sigma_n,1-\sigma_n]\;$ a uniform mesh with
$\;\frac{N}{2}\;$ mesh-intervals is placed, on each of the
sub-intervals
$\;(\sigma_r,\sigma_{r+1}]\;\tx{and}\;(1-\sigma_{r+1},1-\sigma_r],\;\;r=1,\dots,n-1,\;$
a uniform mesh of $\;\frac{N}{2^{n-r+2}}\;$ mesh-intervals is placed
and on both of the sub-intervals $\;[0,\sigma_1]\;$ and
$\;(1-\sigma_1,1]\;$ a uniform mesh of $\;\frac{N}{2^{n+1}}\;$
mesh-intervals is placed. In practice it is convenient to take
\begin{equation}\label{meshpts2} N=2^{n+p+1} \end{equation} for some
natural number $p$. It follows that in the sub-interval
$[\sigma_{r-1},\sigma_{r}]$ there are $N/2^{n-r+3}=2^{r+p-2}$
mesh-intervals. This construction leads to a class of $2^n$
piecewise uniform Shishkin meshes $\Omega^{M,N}$. Note that these meshes are not the same as those
constructed in \cite{GLOR}\\
The following notation is introduced: $h_j=x_{j}-x_{j-1},\; J= \{x_j: D^+ h_j=h_{j+1}-h_j \neq 0\}$. Clearly, $J$ is the set of points at which the mesh-size
changes.  Let $R=\{r:\sigma_r \in J \}$. From the above construction it follows that $J$ is a subset of the set of transition points $\{\sigma_r \}_{r=1}^n
\cup\{1-\sigma_r\}_{r=1}^n$. It is not hard to see that for each point $x_j$ in the mesh-interval $(\sigma_{r-1} ,\sigma_r ]$,
\begin{equation}\label{geom7} h_j
=2^{n-r+3}N^{-1}(\sigma_r-\sigma_{r-1})\end{equation} and so the change in the mesh-size at the point $\sigma_r$ is
\begin{equation}\label{geom8} D^+ h_r=2^{n-r+3}(d_r
-d_{r-1}), \end{equation} where $d_r =\frac{\sigma_{r+1}}{2}-\sigma_r$ for
 $1 \leq r \leq n$, with the conventions $d_0 =0,\; \sigma_{n+1}=1/2.$ Notice that $d_r
\ge 0$, that $\Omega^{M,N}$ is a classical uniform mesh when $d_r = 0$ for all $ r=1 \; \dots \; n$ and, from \eqref{geom8},
 that \begin{equation}\label{geom5}
D^+h_r <0 \;\mathrm{if}\; d_r=0.\end{equation}
 \\
Furthermore
\begin{equation}\label{geom2}\sigma_r \leq C \sqrt{\eps_r} \ln N,
 \; \;\; 1 \leq r \leq n, \end{equation}
and, using (\ref{geom7}), (\ref{geom2}),
\begin{equation}\label{geom4}h_{r}+h_{r+1} \leq
CN^{-1}\ln N \left\{ \begin{array}{l}\;\; \sqrt{\eps_{r+1}}, \;\;\mathrm{if}\;\; D^+ h_r >0, \\
\;\;  \sqrt{\eps_{r}}, \;\;\mathrm{if}\;\; D^+ h_r <0.\end{array}\right .\end{equation}\\
Also
\begin{equation}\label{geom-1}
\sigma_r=2^{-(s-r+1)}\sigma_{s+1}\;\mathrm{when}\; d_r=\dots =d_s
=0, \; 1 \leq r \leq s \leq n.
\end{equation}
The geometrical results in the following lemma are used later.
\begin{lemma} \label{s1} Assume that $d_r>0$ and let $0<s \leq 2$.  Then the following inequalities hold
\begin{equation}\label{geom0}
B^L_r(\sigma_r)=N^{-2}.
\end{equation}
\begin{equation}\label{geom9} x^{(s)}_{r-1,r}\;\leq\;\sigma_r-h_r \;\mathrm{for} \;
1<r\leq n.\end{equation}
\begin{equation}\label{geom10}
\frac{B_{q}^{L}(\sigma_{r})}{\eps_q^s}\leq \frac{1}{\eps_r^s} \;\;
\mathrm{for} \;\; 1 \leq q \leq n, \;\;1 \leq r \leq n.
\end{equation}
\begin{equation}\label{geom3} B_q^L(\sigma_r-h_r)\leq
CB_q^L(\sigma_r)\;\;\mathrm{for}\;\; 1 \leq r \leq q \leq n.
\end{equation}
\end{lemma}
\begin{proof} The proof of \eqref{geom0} follows immediately from
the definition of $\sigma_r$ and the assumption that $d_r>0$.\\ To verify (\ref{geom9}) note that, by Lemma \ref{layers} and \eqref{meshpts2},
\[x^{(s)}_{r-1,r} < 2s\sqrt{\frac{\eps_r}{\alpha}} =
\frac{s\sigma_r}{\ln N} = \frac{s\sigma_r}{(n+p+1)\ln 2} \leq
\frac{\sigma_r}{2}.
\]
Also, by \eqref{meshpts2} and \eqref{geom7}, \[h_r
=2^{n-r+3}N^{-1}(\sigma_{r}-\sigma_{r-1})=2^{2-r-p}(\sigma_{r}-\sigma_{r-1})\leq
\frac{\sigma_{r}-\sigma_{r-1}}{2} < \frac{\sigma_r}{2}.\] It follows
that $x^{(s)}_{r-1,r}+h_r \leq \sigma_r$ as required.
\\
To verify (\ref{geom10}) note that if $q \ge r$ the result is
trivial. On the other hand, if $q<r$, by (\ref{geom9}) and Lemma
\ref{layers},
\[\frac{B_{q}^{L}(\sigma_{r})}{\eps_q^s}\leq
\frac{B_{q}^{L}(x^{(s)}_{q,r})}{\eps_q^s}=
\frac{B_{r}^{L}(x^{(s)}_{q,r})}{\eps_r^s}\leq \frac{1}{\eps_r^s}.\]
Finally, to verify (\ref{geom3}) note, from \eqref{geom7}, that
\begin{equation*}
h_r=2^{n-r+3}N^{-1}(\sigma_{r}-\sigma_{r-1}) \leq
2^{n-r+3}N^{-1}\sigma_r=2^{n-r+4}\sqrt{\frac{\eps_r}{\alpha}}
N^{-1}\ln N.
\end{equation*}
But
\begin{equation*}
e^{2^{n-r+4} N^{-1}\ln N}=(N^{\frac{1}{N}})^{2^{n-r+4}} \leq C,
\end{equation*}
so
\begin{equation*}
\sqrt{\frac{\alpha}{\eps_q }}h_r \leq
\sqrt{\frac{\eps_r}{\eps_q}}2^{n-r+4} N^{-1}\ln N \leq 2^{n-r+4}
N^{-1}\ln N \leq C,
\end{equation*}  since $r\leq q$.
It follows that
\begin{equation*}
B^L _q (\sigma_r -h_r )=B^L _q
(\sigma_r)e^{\sqrt{\frac{\alpha}{\eps_q }}h_r } \leq CB^L _q
(\sigma_r)
\end{equation*} as required. \eop \end{proof}
\section{The discrete problem}
In this section a classical finite difference operator with an
appropriate Shishkin mesh is used to construct a numerical method
for (\ref{BVP}), which is shown later to be essentially second order
parameter-uniform. It is assumed henceforth that the problem data
satisfy whatever smoothness conditions are required.\\
The discrete initial-boundary value problem is now defined on any mesh by the finite difference method
\begin{equation}\label{discreteBVP}
D^-_t\vec{U}-E\delta^2_x\vec{U} +A\vec{U}=\vec{f}\;\; \mathrm{on}
\;\; \Omega^{M,N},\;\; \vec{U}=\vec{u} \;\; \mathrm{on} \;\;
\Gamma^{M,N}.
\end{equation}
This is used to compute numerical approximations to the exact
solution of (\ref{BVP}). Note that (\ref{discreteBVP}),
 can also be written in the operator form
\[\vec{L}^{M,N} \vec{U}\;=\;\vec{f} \;\; \mathrm{on}
\;\; \Omega^{M,N},\;\; \vec{U}=\vec{u} \;\; \mathrm{on} \;\;
\Gamma^{M,N},\] where \[\vec{L}^{M,N} \;=\; D^-_t-E\delta^2_x+A\]
and $D^-_t,\; \delta^2_x,\; D^+_x \; \tx{and} \; D^{-}_x$ are the
difference operators
\[D^-_t\vec{U}(x_j,t_k)\;=\;\dfrac{\vec{U}(x_j,t_k)-\vec{U}(x_{j},t_{k-1})}{t_k-t_{k-1}},\]
\[\delta^2_x\vec{U}(x_j,t_k)\;=\;\dfrac{D^+_x\vec{U}(x_j,t_k)-D^-_x\vec{U}(x_j,t_k)}{(x_{j+1}-x_{j-1})/2},\]
\[D^+_x\vec{U}(x_j,t_k)\;=\;\dfrac{\vec{U}(x_{j+1},t_k)-\vec{U}(x_j,t_k)}{x_{j+1}-x_j},\]
\[D^-_x\vec{U}(x_j,t_k)\;=\;\dfrac{\vec{U}(x_j,t_k)-\vec{U}(x_{j-1},t_k)}{x_j-x_{j-1}}.\]
The following discrete results are analogous to those for the
continuous case.
\begin{lemma}\label{dmax} Let $A(x,t)$ satisfy (\ref{a1}) and (\ref{a2}).
Then, for any mesh function $\vec{\Psi}$, the inequalities $\vec
{\Psi}\;\ge\;\vec 0 \;\mathrm{on}\; \Gamma^{M,N}$ and $\vec{L}^{M,N}
\vec{\Psi}\;\ge\;\vec 0\;$ on $\Omega^{M,N}$ imply that $\;\vec
{\Psi}\ge \vec 0\;$ on $\overline{\Omega}^{M,N}.$
\end{lemma}
\begin{proof} Let $i^*, j^*, k^*$ be such that
$\Psi_{i^*}(x_{j^{*}}, t_{k^{*}})=\min_i\min_{j,k}\Psi_i(x_j,t_k)$
and assume that the lemma is false. Then $\Psi_{i^*}(x_{j^{*}},
t_{k^{*}})<0$ . From the hypotheses we have $j^*\neq 0, \;N$ and
$\Psi_{i^*} (x_{j^*},t_{k^{*}})-\Psi_{i^*}(x_{j^*-1},t_{k^{*}})\leq
0, \; \Psi_{i^*}
(x_{j^*+1},t_{k^{*}})-\Psi_{i^*}(x_{j^*},t_{k^{*}})\geq 0,$ so
$\;\delta^2_x\Psi_{i*}(x_{j*},t_{k^{*}})\;>\;0.\;$ It follows that
\[\ds\left(\vec{L}^N\vec{\Psi}(x_{j*},t_{k^{*}})\right)_{i*}\;=
\;-\eps_{i*}\delta^2_x\Psi_{i*}(x_{j*},t_{k^{*}})+\ds{\sum_{q=1}^n}
a_{i*,\;q}(x_{j*},t_{k^{*}})\Psi_{q}(x_{j*},t_{k^{*}})\;<\;0,\]
which is a contradiction, as required. \eop \end{proof}

An immediate consequence of this is the following discrete stability
result.
\begin{lemma}\label{dstab} Let $A(x,t)$ satisfy (\ref{a1}) and (\ref{a2}).
Then, for any mesh function $\vec{\Psi} $ on $\Omega$,
\[\parallel\vec{\Psi}(x_j,t_k)\parallel\;\le\;\max\left\{||\vec{\Psi}||_{\Gamma^{M,N}}, \frac{1}{\alpha}||
\vec{L}^{M,N}\vec{\Psi}||\right\}. \]
\end{lemma}
\begin{proof} Define the two functions
\[\vec{\Theta}^{\pm}(x_j,t_k)=\max\{||\vec{\Psi}||_{\Gamma^{M,N}},
\frac{1}{\alpha}||\vec{L^{M,N}}\vec{\Psi}||\}\vec{e}\pm
\vec{\Psi}(x_j,t_k)\]where $\vec{e}=(1,\;\dots \;,1)$ is the unit
vector. Using the properties of $A$ it is not hard to verify that
 $\vec{\Theta}^{\pm}\geq \vec{0}$ on $\Gamma^{M,N}$ and
$\vec{L}^{M,N}\vec{\Theta}^{\pm}\geq \vec{0}$ on $\Omega^{M,N}$. It follows from Lemma \ref{dmax} that $\vec{\Theta}^{\pm}\geq \vec{0}$ on
$\overline{\Omega}^{M,N}$.\eop
\end{proof}
The following comparison result will be used in the proof of the
error estimate.
\begin{lemma}\label{comparison}(Comparison Principle) Assume that,
for each $i=1,\; \dots \;,n$,
the mesh functions $\vec{\Phi}$ and  $\vec{Z}$ satisfy
\[|Z_i| \leq \Phi_i,\;\; \mathrm{on}\;\;\Gamma^{M,N}\;\; \mathrm{and}\;\;
|(\vec{L}^{M,N}\vec{Z})_i| \leq (\vec{L}^{M,N} \vec{\Phi})_i\;\; \mathrm{on}\;\;\Omega^{M,N}.\] Then, for each
 $i=1,\; \dots \;,n$,
\[|Z_i| \leq \Phi_i.\]
\end{lemma}
\begin{proof} Define the two mesh functions $\vec{\Psi}^{\pm}$ by
\[\vec{\Psi}^{\pm}=\vec{\Phi} \pm \vec{Z}.\] Then, for each $i=1,\; \dots \;,n$,
satisfies
 \[(\Psi^{\pm})_i\ge 0, \;\; \mathrm{on}\;\;\Gamma^{M,N}\;\; \mathrm{and}\;\;
|(\vec{L}^{M,N}\vec{Z})_i| \leq (\vec{L}^{M,N} \vec{\Phi})_i\;\; \mathrm{on}\;\;\Omega^{M,N}.\] The result follows
 from an application of Lemma \ref{dmax}.
\eop \end{proof}
\section{The local truncation error}
From Lemma \ref{dstab}, it is seen that in order to bound the error
$||\vec{U}-\vec{u}||$ it suffices to bound
$\vec{L}^{M,N}(\vec{U}-\vec{u})$. But this expression satisfies
\begin{equation*}\begin{array}{l}
\vec{L}^{M,N}(\vec{U}-\vec{u})=\vec{L}^{M,N}(\vec{U})-\vec{L}^{M,N}(\vec{u})=\\
\vec{f}-\vec{L}^{M,N}(\vec{u})
=\vec{L}(\vec{u})-\vec{L}^{M,N}(\vec{u})
=(\vec{L}-\vec{L}^{M,N})\vec{u}.\end{array}\end{equation*} It
follows that
\begin{equation*}
\vec{L}^{M,N}(\vec{U}-\vec{u})=(\frac{\partial}{\partial
t}-D^-_t)\vec{u}-E(\frac{\partial^2}{\partial
x^2}-\delta^2_x)\vec{u}.
\end{equation*}
Let $\vec{V}, \vec{W}^L, \vec{W}^R$ be the discrete analogues of
$\vec{v}, \vec{w}^L, \vec{w}^R$ respectively. Then, similarly,
\begin{equation*}\vec{L}^{M,N}(\vec{V}-\vec{v})=(\frac{\partial}{\partial
t}-D^-_t)\vec{v}-E(\frac{\partial^2}{\partial x^2}-\delta^2_x)\vec{v},\end{equation*}
\begin{equation*}\vec{L}^{M,N}(\vec{W}^L-\vec{w}^L)=(\frac{\partial}{\partial
t}-D^-_t)\vec{w}^L-E(\frac{\partial^2}{\partial
x^2}-\delta^2_x)\vec{w}^L,\end{equation*}
\begin{equation*}\vec{L}^{M,N}(\vec{W}^R-\vec{w}^R)=(\frac{\partial}{\partial
t}-D^-_t)\vec{w}^R-E(\frac{\partial^2}{\partial
x^2}-\delta^2_x)\vec{w}^R,\end{equation*} and so, for each
$i=1,\;\dots \;,n$,
\begin{equation}\label{LTEV}
|(\vec{L}^{M,N}(\vec{V}-\vec{v}))_i|\leq |(\frac{\partial}{\partial t}-D^-_t)v_i|+
|\eps_i(\frac{\partial^2}{\partial x^2}-\delta^2_x)v_i|,
\end{equation}
\begin{equation}\label{LTEWL}
|(\vec{L}^{M,N}(\vec{W^L}-\vec{w^L}))_i|\leq
|(\frac{\partial}{\partial
t}-D^-_t)w^L_i|+|\eps_i(\frac{\partial^2}{\partial
x^2}-\delta^2_x)w^L_i|,
\end{equation}
\begin{equation}\label{LTEWR}
|(\vec{L}^{M,N}(\vec{W^R}-\vec{w^R}))_i|\leq
|(\frac{\partial}{\partial
t}-D^-_t)w^R_i|+|\eps_i(\frac{\partial^2}{\partial
x^2}-\delta^2_x)w^R_i|.
\end{equation}
Thus, the smooth and singular components of the local truncation
error can be treated separately. Note that, for any smooth
function $ \psi $, the following
distinct estimates of the local truncation error hold:\\
for each$(x_j,t_k)\in \Omega^{L,M}$
\begin{equation}\label{lte0}
|(\frac{\partial}{\partial t}-D^-_t)\psi(x_j,t_k)|\;\le\;
C(t_k-t_{k-1})\max_{s\;\in\;[t_{k-1},t_k]}|\frac{\partial^2\psi}{\partial
t^2}(x_j,s)|,
\end{equation}
\begin{equation}\label{lte1}
|(\frac{\partial^2}{\partial x^2}-\delta^2_x)\psi(x_j,t_k)|\;\le\;
C\max_{s\;\in\;I_j}|\frac{\partial^2\psi}{\partial x^2}(s,t_k)|,
\end{equation}
and
\begin{equation}\label{lte2}
|(\frac{\partial^2}{\partial
x^2}-\delta^2_x)\psi(x_j,t_k)|\;\le\;C(x_{j+1}-x_{j-1}) \max_{s\in
I_j}|\frac{\partial^3\psi}{\partial x^3}(s,t_k)|.
\end{equation}
Assuming, furthermore, that $x_j \notin J$, then
\begin{equation}\label{lte3}
|(\frac{\partial^2}{\partial
x^2}-\delta^2_x)\psi(x_j,t_k)|\;\le\;C(x_{j+1}-x_{j-1})^2
\max_{s\in I_j}|\frac{\partial^4\psi}{\partial x^4}(s,t_k)|.
\end{equation}
Here $I_j=[x_{j-1}, x_{j+1}]$.
\section{Error estimate}
The proof of the error estimate is broken into two parts. In the
first a theorem concerning the smooth part of the error is proved.
Then the singular part of the error is considered. A barrier
function is now constructed, which is used in both parts of the
proof.\\
For each $r \in R$, introduce a piecewise linear polynomial
$\theta_r$ on $\overline{\Omega}$, defined by
\begin{equation*} \theta_r(x)=
\left\{ \begin{array}{l}\;\; \dfrac{x}{\sigma_r}, \;\; 0 \leq x \leq \sigma_r. \\
\;\; 1, \;\; \sigma_r < x < 1-\sigma_r.
\\ \;\; \dfrac{1-x}{\sigma_r}, \;\; 1-\sigma_r \leq x \leq 1.  \end{array}\right .\end{equation*}\\
It is not hard to verify that
\begin{equation}\label{theta} L^{M,N}(\theta_r(x_j)\vec{e})_i \ge
\left\{ \begin{array}{l}\;\;
\alpha \theta_{r}(x_j), \;\;  \mathrm{if} \;\;x_j \notin J \\
\;\; \alpha+\dfrac{2\eps_i}{ \sigma_r
(h_r+h_{r+1})}, \;\;  \mathrm{if} \;\;x_j=\sigma_r \in J. \end{array}\right .\end{equation}\\
On the Shishkin mesh $\Omega^{M,N}$ define the barrier function
$\vec{\Phi}$ by
\begin{equation}\label{barrier}\vec{\Phi}(x_j,t_k)=C[M^{-1}+N^{-2}+
N^{-2}(\ln N)^3\ds\sum_{r\in
R}\theta_r (x_j)]\vec{e},\end{equation} where $C$ is any sufficiently large constant.\\
Then, on $\Omega^{M,N}$,  $\vec{\Phi}$ satisfies
\begin{equation}\label{barrierbound1}0 \leq \Phi_{i}(x_j,t_k) \leq
C(M^{-1}+N^{-2}(\ln N)^{3}),\;\; 1 \leq i \leq n.\end{equation}
Also, for $x_j \notin J$,
\begin{equation}\label{barrierbound3}
(L^{M,N}\vec{\Phi})_i(x_j,t_k) \ge C(M^{-1}+N^{-2}(\ln
N)^{3})\end{equation} and, for $x_j \in J$, using
(\ref{geom4}),(\ref{theta}),
\begin{equation}\label{barrierbound4}
(L^{M,N}\vec{\Phi}(x_j,t_k))_i \ge
\left\{\begin{array}{l}C(M^{-1}+N^{-2}+
\frac{\eps_i}{\sqrt{\eps_r \eps_{r+1}}} N^{-1} \ln N),\;\;\mathrm{if}\;\;D^+ h_r>0, \\
C(M^{-1}+N^{-2}+\frac{\eps_i}{\eps_r} N^{-1} \ln N),\;\;
\mathrm{if}\;\;D^+ h_r<0.\end{array} \right .\end{equation} The
following theorem gives the required error estimate for the smooth
component.
\begin{theorem}\label{smootherrorthm} Let $A(x,t)$ satisfy (\ref{a1}) and
(\ref{a2}). Let $\vec v$ denote the smooth component of the exact
solution from (\ref{BVP}) and $\vec V$ the smooth component of the
 discrete solution from (\ref{discreteBVP}).  Then
\begin{equation}\;\; ||\vec{V}-\vec{v}|| \leq C(M^{-1}+N^{-2}(\ln N)^3). \end{equation}
\end{theorem}
\begin{proof}
It suffices to show that
\begin{equation}\label{ratiobound} \frac{|(L^{M,N}(\vec{V}-\vec{v}))_i (x_j,t_k)|}{|(L^{M,N}\vec{\Phi})_i
(x_j,t_k)|} \leq C,
\end{equation}
for each $i=1,\;\dots \;,n,$ because an application of the
Comparison
Principle then yields the required result.\\ For each mesh point $x_j$ either $x_j \notin J$ or $x_j \in J$.\\
Suppose first that $x_j \notin J$. Then, from \eqref{barrierbound3},
\begin{equation}\label{smoothdenom}
(L^{M,N}\vec{\Phi}(x_j,t_k))_i \ge C(M^{-1}+N^{-2})\end{equation}
and from \eqref{lte0}, \eqref{lte3} and Lemma \ref{lsmooth1}
\begin{equation}\label{smoothnumer}
\begin{array}{lcl}|(L^{M,N}(\vec{V}-\vec{v}))_i(x_j,t_k)|&\leq &C(t_k-t_{k-1}+(x_{j+1}-x_{j-1})^2)\\  &\leq &C(M^{-1}+(h_j+h_{j+1})^2)\\
&\leq & C(M^{-1}+N^{-2}).\end{array}\end{equation} Then
\eqref{ratiobound} follows from \eqref{smoothdenom} and
\eqref{smoothnumer} as required.\\
On the other hand, when $x_j \in J$, by \eqref{lte0}, \eqref{lte2}
and Lemma \ref{lsmooth2}
\begin{equation}\label{s2}|(L^{M,N}(\vec{V}-\vec{v}))_i(x_j,t_k)|
\leq C[M^{-1}+
 \eps_i(h_r+h_{r+1})(1+\sum^n_{q=i}\frac{B_q (\sigma_r
 -h_r)}{\sqrt{\eps_q}})].\end{equation}
The cases $i \ge r$ and $i<r$
are treated separately.\\
Suppose first that $i \ge r$, then it is not hard to see that
\begin{equation}\label{jumpnumer1}\begin{array}{l}
|(L^{M,N}(\vec{V}-\vec{v}))_i(x_j,t_k)| \\
\leq C[M^{-1}+\eps_i(h_r+h_{r+1})(1+\frac{1}{\sqrt{\eps_i}})]
\\ \leq C[M^{-1}+(h_r+h_{r+1})\sqrt{\eps_i}] .\end{array}\end{equation}
Combining \eqref{barrierbound4} and \eqref{jumpnumer1},
\eqref{ratiobound} follows using \eqref{geom4} and the ordering of
the $\eps_i$.\\
On the other hand, if $i<r$ then $\eps_i \leq \eps_{r-1} < \eps_r$.
Also, either $d_r>0$ or $d_r=0.$\\
First, suppose that $d_r>0$. Then, by Lemma \ref{s1}, \[\sigma_r-h_r
\ge x^{(\frac{1}{2})}_{q,r} ;\; \mathrm{for} \;\; i \leq q \leq
r-1\] and so, by Lemma \ref{layers}\[\sum^{r-1}_{q=i}\frac{B_q
(\sigma_r
 -h_r)}{\sqrt{\eps_q}} \leq C\frac{B_r (\sigma_r
 -h_r)}{\sqrt{\eps_r}}.\] Combining this with \eqref{s2} gives
\begin{equation}\label{jumpnumer2}
|(L^{M,N}(\vec{V}-\vec{v}))_i(x_j,t_k)| \\
\leq C[M^{-1}+\frac{\eps_i}{\sqrt{\eps_r}}(h_r+h_{r+1})]
.\end{equation} Combining \eqref{barrierbound4} and
\eqref{jumpnumer2}, \eqref{ratiobound} follows using \eqref{geom4}
and the ordering of
the $\eps_i$.\\
Secondly, suppose that $d_r=0$. Then $d_{r-1}>0$ and $D^+ h_r<0$.
Then, by Lemma \ref{s1} with $r-1$ instead of $r$,  \[\sigma_r-h_r
\ge \sigma_{r-1}>\sigma_{r-1}-h_{r-1} \ge
x^{(\frac{1}{2})}_{q,{r-1}}\;\; \mathrm{for}\;\; i \leq q \leq r-2\]
 and so, by Lemma
\ref{layers}\[\sum^{r-2}_{q=i}\frac{B_q (\sigma_r
 -h_r)}{\sqrt{\eps_q}} \leq C\frac{B_{r-1} (\sigma_r
 -h_r)}{\sqrt{\eps_{r-1}}}\leq
 C\frac{B_{r-1}(\sigma_{r-1})}{\sqrt{\eps_{r-1}}}=C\frac{N^{-2}}{\sqrt{\eps_{r-1}}}.\]
Combining this with \eqref{s2} and \eqref{geom4} gives
\begin{equation}\label{jumpnumer3}\begin{array}{l}
|(L^{M,N}(\vec{V}-\vec{v}))_i(x_j,t_k)| \\
\leq C[M^{-1}+\eps_i(h_r+h_{r+1})(\frac{N^{-2}}{\sqrt{\eps_{r-1}}}+
\frac{1}{\sqrt{\eps_{r}}})]\\
\leq C[M^{-1}+\sqrt{\eps_i\eps_r}N^{-3}\ln N+\eps_i N^{-1}\ln N].
\end{array}\end{equation}
Combining \eqref{barrierbound4} and \eqref{jumpnumer3},
\eqref{ratiobound} follows using the ordering of
the $\eps_i$ and noting that in this case
the middle term in the denominator is used to bound the
middle term in the numerator.\\
\eop\end{proof} Before the singular part of the error is estimated
the following lemmas are established.
\begin{lemma}\label{est1} Assume that $x_j \notin J$.  Let $A(x,t)$ satisfy
 (\ref{a1}) and (\ref{a2}). Then,
on $\Omega^{M,N}$, for each $1 \leq i \leq n$, the following
estimates hold
\begin{equation} |(L^{M,N}(\vec{W^L}-\vec{w^L}))_i(x_j,t_k)|\leq C(M^{-1}+\frac{(x_{j+1}-x_{j-1})^2}{\eps_1}).\end{equation} An analogous result holds for the $w^R_i$.
\end{lemma}
\begin{proof} Since $x_j \notin J$, from (\ref{lte3}) and Lemma \ref{lsingular},
it follows that
\begin{equation*}\begin{array}{lcl}|(L^{M,N}(\vec{W^L}-\vec{w^L}))_i(x_j,t_k)|&=&
|((\frac{\partial}{\partial t}-D^-_t)-E(\frac{\partial^2}{\partial
x^2}-\delta^2_x))\vec{w}^L _i(x_j,t_k)|\\ 
& \leq &
C(M^{-1}+(x_{j+1}-x_{j-1})^2\;\ds\max_{s\;\in\;I_j}\ds\sum_{q\;=\;1}^n
\dfrac{B^{L}_{q}(s)}{\eps_q}) \\ & \leq &
C(M^{-1}+\frac{(x_{j+1}-x_{j-1})^2}{\eps_1})\end{array}\end{equation*}
as required.\eop\end{proof}

The following decompositions are introduced
\[w^L_i=\sum_{m=1}^{r+1}w_{i,m},\] where the components 
are defined by
\[w_{i,r+1}=\left\{ \begin{array}{ll} p^{(s)}_{i} & {\rm on}\;\;[0,x^{(s)}_{r,r+1})\\
 w^L_i & {\rm otherwise} \end{array}\right. \]
and for each $m$,  $r \ge m \ge 2$,
\[w_{i,m}=\left\{ \begin{array}{ll} p^{(s)}_{i} & \rm{on} \;\; [0,x^{(s)}_{m-1,m})\\
w^L_i-\ds\sum_{q=m+1}^{r+1} w_{i,q} & {\rm otherwise}
\end{array}\right. \]
and
\[w_{i,1}=w^L_i-\sum_{q=2}^{r+1} w_{i,q}\;\; \rm{on} \;\; [0,1]. \]
Here the polynomials $p^{(s)}_{i}$, for $s=3/2$ and $s=1$, are
defined by
\\ \[ p^{(3/2)}_{i}(x,t)=\sum_{q=0}^3
\frac{\partial^q w^{(L)}_i}{\partial
x^q}(x^{(3/2)}_{r,r+1},t)\frac{(x-x^{(3/2)})^q}{q!}\]\\ and \\
\[ p^{(1)}_{i}(x,t)=\sum_{q=0}^4
\frac{\partial^q w^{(L)}_i}{\partial
x^q}(x^{(1)}_{r,r+1},t)\frac{(x-x^{(1)})^q}{q!}.\]

\begin{lemma}\label{general} Assume that $d_r>0$. Let $A(x,t)$ satisfy (\ref{a1}) and (\ref{a2}). Then, for
each $1 \leq i \leq n$,   there exists a decomposition
\[ w^L_i=\sum_{q=1}^{r+1}w_{i,q}, \] for which the following
estimates hold for each $q$ and $r$,  $1 \le q \le r$,
\[|\frac{\partial^2 w_{i,q}}{\partial x^2}(x_j,t_k)| \leq C \min \{\frac{1}{\eps_q}, \frac{1}{\eps_i}\}B^L_q(x_j),\]
\[|\frac{\partial^3 w_{i,q}}{\partial x^3}(x_j,t_k)| \leq C \min \{\frac{1}{\eps_q^{3/2}}, \frac{1}{\eps_i \sqrt{\eps_q}}\}B^L_q(x_j),\]
\[|\frac{\partial^3 w_{i,r+1}}{\partial x^3}(x_j,t_k)| \leq C \min \{\sum_{q=r+1}^n \frac{B^L_q(x_j)}{\eps_q^{3/2}}, \sum_{q=r+1}^n \frac{B^L_q(x_j)}{\eps_i \sqrt{\eps_q}}\},\]
\[|\frac{\partial^4 w_{i,q}}{\partial x^4}(x_j,t_k)| \leq C \frac{B^L_q(x_j)}{\eps_i \eps_q},\]
\[|\frac{\partial^4 w_{i,r+1}}{\partial x^4}(x_j,t_k)| \leq
C\sum_{q=r+1}^n \frac{B^L_q(x_j)}{\eps_i \eps_q}.\]
Analogous results hold for the $w^R_i$ and their derivatives.
\end{lemma}
\begin{proof}
First consider the decomposition corresponding $s=3/2$.\\
From the above definitions it follows that, for each $m$, $1 \leq m
\leq r$,
$w_{i,m}=0 \;\; \rm{on} \;\; [x^{(3/2)}_{m,m+1},1]$.\\

To establish the bounds on the third derivatives it is seen that:
for $x \in [x^{(3/2)}_{r,r+1},1]$, Lemma \ref{lsingular} and $x \geq
x^{(3/2)}_{r,r+1}$ imply that
\[|\frac{\partial^3 w_{i,r+1}}{\partial x^3}(x,t)| =
|\frac{\partial^3 w^L_{i}}{\partial x^3}(x,t)| \leq
C\sum_{q=1}^n \frac{B^L_q(x)}{\eps_q^{3/2}} \leq C\sum_{q=r+1}^n
\frac{B^L_q(x)}{\eps_q^{3/2}};\]

for $x \in [0, x^{(3/2)}_{r,r+1}]$, Lemma \ref{lsingular} and $x
\leq x^{(3/2)}_{r,r+1}$ imply that
\[|\frac{\partial^3
w_{i,r+1}}{\partial x^3}(x,t)| = |\frac{\partial^3 w^L_{i}}{\partial
x^3}(x^{(3/2)}_{r,r+1},t)|\] \[\leq \sum_{q=1}^{n}
\frac{B^L_q(x^{(3/2)}_{r,r+1})}{\eps_q^{3/2}} \leq \sum_{q=r+1}^{n}
\frac{B^L_q(x^{(3/2)}_{r,r+1})}{\eps_q^{3/2}} \leq \sum_{q=r+1}^{n}
\frac{B^L_q(x)}{\eps_q^{3/2}};\]

and for each $m=r, \;\; \dots \;\;,2$, it follows that\\

for $x \in [x^{(3/2)}_{m,m+1},1]$, $w_{i,m}^{(3)}=0;$

for $x \in [x^{(3/2)}_{m-1,m},x^{(3/2)}_{m,m+1}]$, Lemma
\ref{lsingular} implies that
\[|\frac{\partial^3
w_{i,m}}{\partial x^3}(x,t)| \leq |\frac{\partial^3
w^L_{i}}{\partial x^3}(x,t)|+\sum_{q=m+1}^{r+1}|\frac{\partial^3
w_{i,q}}{\partial x^3}(x,t)|\] \[ \leq C\sum_{q=1}^n
\frac{B^L_q(x)}{\eps_q^{3/2}} \leq C\frac{B^L_m(x)}{\eps_m^{3/2}};\]

for $x \in [0, x^{(3/2)}_{m-1,m}]$, Lemma \ref{lsingular} and $x
\leq x^{(3/2)}_{m-1,m}$ imply that
\[|\frac{\partial^3
w_{i,m}}{\partial x^3}(x,t)| =|\frac{\partial^3 w^L_{i}}{\partial
x^3}(x^{(3/2)}_{m-1,m},t)|\] \[ \leq C\sum_{q=1}^n
\frac{B^L_q(x^{(3/2)}_{m-1,m})}{\eps_q^{3/2}} \leq
C\frac{B^L_m(x^{(3/2)}_{m-1,m})}{\eps_m^{3/2}} \leq
C\frac{B^L_m(x)}{\eps_m^{3/2}};
\]

for $x \in [x^{(3/2)}_{1,2},1],\;\; \frac{\partial^3
w_{i,1}}{\partial x^3}=0;$

for $x \in [0, x^{(3/2)}_{1,2}]$, Lemma \ref{lsingular} implies that
\[|\frac{\partial^3
w_{i,1}}{\partial x^3}(x,t)| \leq |\frac{\partial^3
w^L_{i}}{\partial x^3}(x,t)|+\sum_{q=2}^{r+1}|\frac{\partial^3
w_{i,q}}{\partial x^3}(x,t)| \leq C\sum_{q=1}^n
\frac{B^L_q(x)}{\eps_q^{3/2}} \leq C\frac{B^L_1(x)}{\eps_1^{3/2}}.\]

For the bounds on the second derivatives note that, for each $m$, $1
\leq m \leq r $ :

for $x \in [x^{(3/2)}_{m,m+1},1],\;\; \frac{\partial^2
w_{i,m}}{\partial x^2}=0;$

for $x \in [0, x^{(3/2)}_{m,m+1}],\;\; \int_x^{x^{(3/2)}_{m,m+1}}
\frac{\partial^3 w_{i,m}}{\partial x^3}(s,t)ds =\\ \frac{\partial^2
w_{i,m}}{\partial x^2}(x^{(3/2)}_{m-1,m},t)- \frac{\partial^2
w_{i,m}}{\partial x^2}(x,t)= -\frac{\partial^2
w_{i,m}}{\partial x^2}(x,t)$ \\
and so
\[|\frac{\partial^2
w_{i,m}}{\partial x^2}(x,t)| \leq
\int_x^{x^{(3/2)}_{m,m+1}}|\frac{\partial^3 w_{i,m}}{\partial
x^3}(s,t)|ds \leq \frac{C}{\eps_m^{3/2}}\int_{x}^{x^{(3/2)}_{m,m+1}}
B^L_m(s)ds \leq C\frac{B^L_m(x)}{\eps_m}.\] This completes the proof
of the estimates for $s=3/2$.\\
For the  estimates in the case $s=1$ consider the decomposition
\[w^L_i=\sum_{m=1}^{r+1}w_{i,m}.\] 
From the above definitions it follows that, for each $m$, $1 \leq m
\leq r$,
$w_{i,m}=0 \;\; \rm{on} \;\; [x^{(1)}_{m,m+1},1]$.\\
To establish the bounds on the fourth derivatives it is seen that:

for $x \in [x^{(1)}_{r,r+1},1]$, Lemma \ref{lsingular} and $x \geq
x^{(1)}_{r,r+1}$ imply that
\[|\eps_i \frac{\partial^4 w_{i,r+1}}{\partial x^4}(x,t)| =|\eps_i \frac{\partial^4 w^L_{i}}{\partial x^4}(x,t)| \leq
C\sum_{q=1}^n \frac{B^L_q(x)}{\eps_q} \leq C\sum_{q=r+1}^n
\frac{B^L_q(x)}{\eps_q};\]

for $x \in [0, x^{(1)}_{r,r+1}]$, Lemma \ref{lsingular} and $x \leq
x^{(1)}_{r,r+1}$ imply that
\[|\eps_i \frac{\partial^4 w_{i,r+1}}{\partial x^4}(x,t)| =|\eps_i \frac{\partial^4 w^L_{i}}{\partial x^4}(x^{(1)}_{r,r+1},t)|
\] \[\leq \sum_{q=1}^{n} \frac{B^L_q(x^{(1)}_{r,r+1})}{\eps_q} \leq
C\sum_{q=r+1}^{n} \frac{B^L_q(x^{(1)}_{r,r+1})}{\eps_q} \leq
C\sum_{q=r+1}^{n} \frac{B^L_q(x)}{\eps_q};\]

and for each $m=r, \;\; \dots \;\;,2$, it follows that\\

for $x \in [x^{(1)}_{m,m+1},1]$,\;\; $\frac{\partial^4
w_{i,m}}{\partial x^4}=0;$

for $x \in [x^{(1)}_{m-1,m},x^{(1)}_{m,m+1}]$, Lemma \ref{lsingular}
implies that
\[|\eps_i \frac{\partial^4 w_{i,m}}{\partial x^4}(x,t)| \leq |\eps_i \frac{\partial^4 w^L_{i}}{\partial x^4}(x,t)|+\sum_{q=m+1}^{r+1}|\eps_i \frac{\partial^4 w_{i,q}}{\partial x^4}(x,t)|
\] \[\leq C\sum_{q=1}^n \frac{B^L_q(x)}{\eps_q} \leq
C\frac{B^L_m(x)}{\eps_m};\]

for $x \in [0, x^{(1)}_{m-1,m}]$, Lemma \ref{lsingular} and $x \leq
x^{(1)}_{m-1,m}$ imply that
\[|\eps_i \frac{\partial^4 w_{i,m}}{\partial x^4}(x,t)| =
|\eps_i \frac{\partial^4 w^L_{i}}{\partial x^4}(x^{(1)}_{m-1,m},t)|
\] \[\leq C\sum_{q=1}^n \frac{B^L_q(x^{(1)}_{m-1,m})}{\eps_q} \leq
C\frac{B^L_m(x^{(1)}_{m-1,m})}{\eps_m} \leq
C\frac{B^L_m(x)}{\eps_m};
\]

for $x \in [x^{(1)}_{1,2},1],\;\; \frac{\partial^4 w_{i,1}}{\partial
x^4}=0;$

for $x \in [0, x^{(1)}_{1,2}]$, Lemma \ref{lsingular} implies that
\[|\eps_i \frac{\partial^4 w_{i,1}}{\partial
x^4}(x,t)| \leq |\eps_i \frac{\partial^4 w^L_{i}}{\partial
x^4}(x,t)|+\sum_{q=2}^{r+1}|\eps_i \frac{\partial^4
w_{i,q}}{\partial x^4}(x,t)|\] \[\leq C\sum_{q=1}^n
\frac{B^L_q(x)}{\eps_q} \leq C\frac{B^L_1(x)}{\eps_1}.\]

For the bounds on the second and third derivatives note that, for
each $m$, $1 \leq m \leq r $ :

for $x \in [x^{(1)}_{m,m+1},1],\;\; \frac{\partial^2
w_{i,m}}{\partial x^2}=0=\frac{\partial^3 w_{i,m}}{\partial x^3};$

for $x \in [0, x^{(1)}_{m,m+1}],\;\;
\ds\int_x^{x^{(1)}_{m,m+1}}\eps_i \frac{\partial^4 w_{i,m}}{\partial
x^4}(s,t)ds \\= \eps_i \frac{\partial^3 w_{i,m}}{\partial
x^3}(x^{(1)}_{m,m+1},t)- \eps_i \frac{\partial^3 w_{i,m}}{\partial
x^3}(x,t)= -\eps_i \frac{\partial^3 w_{i,m}}{\partial
x^3}(x,t)$ \\
and so
\[|\eps_i \frac{\partial^3 w_{i,m}}{\partial
x^3}(x,t)| \leq \int_x^{x^{(1)}_{m,m+1}}|\eps_i \frac{\partial^4
w_{i,1}}{\partial x^4}(s,t)|ds \\ \leq
\frac{C}{\eps_m}\int_{x}^{x^{(1)}_{m,m+1}} B^L_m(s)ds \leq
C\frac{B^L_m(x)}{\sqrt\eps_m}.\] In a similar way, it can be shown
that
\[|\eps_i \frac{\partial^2 w_{i,m}}{\partial
x^2}(x,t)| \leq C B^L_m(x).\]
The proof for the $w^R_i$ and their derivatives is similar. \eop
\end{proof}
\begin{lemma}\label{general1} Assume that $d_r>0$. Let $A(x)$ satisfy (\ref{a1}) and (\ref{a2}).
Then, for each $i$, $1 \leq i \leq n$, and each $(x_j,t_k) \in
\Omega^{M,N}$
\begin{equation}|\label{gen1}(L^{M,N}(\vec{W^L}-\vec{w}^L)_i(x_j,t_k))| \leq C
[M^{-1}+B^L_r(x_{j-1})+\frac{x_{j+1}-x_{j-1}}{\sqrt{\eps_{r+1}}}]
\end{equation}
and
\begin{equation}\label{gen2}|(L^{M,N}(\vec{W^L}-\vec{w}^L)_i(x_j,t_k))| \leq C[M^{-1}+\eps_i
\sum^r_{q=1}\frac{B^L_{q}(x_{j-1})}{\eps_q}+\frac{\eps_i}{\eps_{r+1}}
\frac{x_{j+1}-x_{j-1}}{\sqrt{\eps_{r+1}}}].\end{equation}
Assuming, furthermore, that $x_j \notin J,$ then
\begin{equation}\label{gen3}|(L^{M,N}(\vec{W^L}-\vec{w}^L)_i(x_j,t_k))| \leq C
[M^{-1}+B^L_r(x_{j-1})+\frac{(x_{j+1}-x_{j-1})^2}{\eps_{r+1}}].
\end{equation}
Analogous results hold for the $W^R-w^R_i$ and their derivatives.
\end{lemma}
\begin{proof} Using \eqref{LTEWL}, \eqref{lte0} and the bound in
Lemma \ref{lsingular}, for any $(x_j,t_k) \in \Omega^{M,N}$,
\begin{equation} \label{ls1}|(L^{M,N}(\vec{W^L}-\vec{w}^L)_i(x_j,t_k))| \leq
C[(t_k-t_{k-1})+|\eps_i(\delta^2_x-\frac{\partial^2}{\partial
x^2})w^L_i(x_j, t_k)| ].
\end{equation}
From the decompositions and bounds in Lemma \ref{general}, with
\eqref{lte1} and \eqref{lte2}, it follows from \eqref{ls1} that
\begin{equation} \begin{array}{l} |\eps_i(\delta^2_x-\frac{\partial^2}{\partial
x^2})w^L_i(x_j, t_k)| \\ \\\leq
C[\sum_{q=1}^{r}|\eps_i(D^2_x-\frac{\partial^2}{\partial
x^2})w_{i,q}(x_j, t_k)|+ |\eps_i(D^2_x-\frac{\partial^2}{\partial
x^2})w_{i,r+1}(x_j, t_k)] \\ \\
\leq C[\sum_{q=1}^{r}\max_{s \in I_j}|\eps_iw_{i,q}^{(2)}(s,t_k)|
+(x_{j+1}-x_{j-1})\max_{s \in
I_j}|\eps_iw_{i,r+1}^{(3)}(s,t_k)|)]\\ \\ \leq
C[\sum_{q=1}^{r}\min\{\frac{\eps_i}{\eps_q},1\}B^L_q
(x_{j-1})+(x_{j+1}-x_{j-1})\min\{\frac{\eps_i}{\eps_{r+1}},1\}\frac{B^L_{r+1}
(x_{j-1})}{\sqrt{\eps_{r+1}}}].
\end{array}\end{equation}
Substituting $1$ for each of the $\min$ expressions gives
\eqref{gen1} and \eqref{gen2} is obtained by substituting the
appropriate ratio $\frac{\eps_i}{\eps_q}$ in each such
expression.\\
In the remaining case when $x_j \notin J$, \eqref{lte3} can be
used instead of \eqref{lte2}, and it follows by a similar argument
to the above that
\begin{equation}  |\eps_i(\delta^2_x-\frac{\partial^2}{\partial
x^2})w^L_i(x_j, t_k)| \leq
C[\sum_{q=1}^{r}\min\{\frac{\eps_i}{\eps_q},1\}B^L_q
(x_{j-1})+\frac{(x_{j+1}-x_{j-1})^2}{\eps_{r+1}}].
\end{equation}
Substituting $1$ for the $\min$ expression, as before, gives
\eqref{gen3}.\\
The proof for the $w^R_i$ and their derivatives is similar.\eop
\end{proof}

\begin{lemma}\label{est3} Let $A(x,t)$ satisfy (\ref{a1}) and (\ref{a2}).
Then, on $\Omega^{M,N}$, for each $1 \leq i \leq n$, the following
estimates hold
\begin{equation} |(L^{M,N}(\vec{W^L}-\vec{w^L}))_i(x_j,t_k)|\leq
C(M^{-1}+B^L_n(x_{j-1})).\end{equation} An analogous result holds for the $w^R_i$.
\end{lemma}
\begin{proof}
From $\;\eqref{lte1}\;$ and Lemma \ref{lsingular}, for each
$\;i=1,\dots,n\;$, it follows that on $\Omega^{M,N},$
\begin{equation*}\begin{array}{lcl}|(L^{M,N}(\vec{W^L}-\vec{w^L}))_i(x_j,t_k)|&=&
|((\frac{\partial}{\partial t}-D^-_t)-E(\frac{\partial^2}{\partial
x^2}-\delta^2_x))\vec{w}^L _i(x_j,t_k)|\\
& \leq &
C(M^{-1}+\eps_i\ds\sum_{q=i}^{n}\dfrac{B^L_q(x_{j-1})}{\eps_q})\\ &
\le & C(M^{-1}+B^L_n(x_{j-1})).\end{array}\end{equation*} The proof
for the $w^R_i$ and their derivatives is similar.\eop
\end{proof}

The following theorem provides the error estimate for the singular
component.
\begin{theorem} \label{singularerrorthm}Let $A(x,t)$ satisfy (\ref{a1}) and
(\ref{a2}). Let $\vec w$ denote the singular component of the exact
solution from (\ref{BVP}) and $\vec W$ the singular component of the
 discrete solution from (\ref{discreteBVP}).  Then
\begin{equation}\;\; ||\vec{W}-\vec{w}|| \leq C(M^{-1}+N^{-2}(\ln N)^{3}). \end{equation}
\end{theorem}

\begin{proof}
Since $\vec{w}=\vec{w}^L+\vec{w}^R$, it suffices to prove the result
for $\vec{w}^L$ and $\vec{w}^R$ separately. Here it is proved for
$\vec{w}^L$ by an application of Lemma \ref{comparison}. A similar proof holds for $\vec{w}^R$.\\
The proof is in two parts: $x_j$ is such that either $x_j \notin J$ or $x_j=\sigma_r \in J$. \\
First assume that $x_j \notin J$.
Each open subinterval $(\sigma_k,\sigma_{k+1})$ is treated separately.\\
First, consider $x_j \in (0,\sigma_1)$. Then, on each mesh $M$,
$x_{j+1}-x_{j-1} \leq CN^{-1}\sigma_1$ and the result follows
 from (\ref{geom2}) and Lemma \ref{est1}.\\
Secondly, consider $x_j \in (\sigma_1,\sigma_2)$, then $\sigma_1
\leq x_{j-1}$ and $x_{j+1}-x_{j-1} \leq CN^{-1}\sigma_2$. The
$2^{n}$ possible meshes are divided into subclasses of two types.
On the meshes $\overline{\Omega}^{M,N}$ with $b_1=0$ the result
follows from (\ref{geom2}), (\ref{geom-1}) and Lemma \ref{est1}.
On the meshes $\overline{\Omega}^{M,N}$ with
$b_1=1$ the result follows from (\ref{geom2}), (\ref{geom0}) and Lemma \ref{general}.\\
Thirdly, in the general case $x_j \in (\sigma_m,\sigma_{m+1})$ for
$2 \leq m \leq n-1$, it follows that $\sigma_m \leq x_{j-1}$ and
$x_{j+1}-x_{j-1} \leq CN^{-1}\sigma_{m+1}$. Then
$\overline{\Omega}^{M,N}$ is divided into subclasses of three types:
$\overline{\Omega}^{M,N}_0=\{\overline{\Omega}^{M,N}: b_1= \dots
=b_m =0\},\; \overline{\Omega}^{M,N}_{r}=\{\overline{\Omega}^{M,N}:
b_r=1, \; b_{r+1}= \dots =b_m =0 \; \mathrm{for \; some}\; 1 \leq r
\leq m-1\}$ and
$\overline{\Omega}^{M,N}_m=\{\overline{\Omega}^{M,N}: b_m=1\}.$ On
$\overline {\Omega}^{M,N}_0$ the result follows from (\ref{geom2}),
(\ref{geom-1}) and Lemma \ref{est1}; on $\overline{\Omega}^{M,N}_r$
from (\ref{geom2}), (\ref{geom-1}), (\ref{geom0}) and Lemma
\ref{general}; on $\overline{\Omega}^{M,N}_m$ from (\ref{geom2}),
(\ref{geom0})
and Lemma \ref{general}. \\
Finally, for $x_j \in (\sigma_n,1)$, $\sigma_n \leq x_{j-1}$ and
$x_{j+1}-x_{j-1} \leq CN^{-1}$. Then $\overline{\Omega}^{M,N}$ is
divided into subclasses of three types:
$\overline{\Omega}^{M,N}_0=\{\overline{\Omega}^{M,N}: b_1= \dots
=b_n =0\},\; \overline{\Omega}^{M,N}_{r}=\{\overline{\Omega}^{M,N}:
b_r=1, \; b_{r+1}= \dots =b_n =0 \; \mathrm{for \; some}\; 1 \leq r
\leq n-1\}$ and
$\overline{\Omega}^{M,N}_n=\{\overline{\Omega}^{M,N}: b_n=1\}.$ On
$\overline{\Omega}^{M,N}_0$ the result follows from (\ref{geom2}),
(\ref{geom-1}) and Lemma \ref{est1}; on $\overline{\Omega}^{M,N}_r$
from (\ref{geom2}), (\ref{geom-1}), (\ref{geom0}) and Lemma
\ref{general}; on $\overline{\Omega}^{M,N}_n$ from (\ref{geom0}) and
Lemma
\ref{est3}.\\\\
Now assume that $x_j \in J$. Then $x_j=\sigma_r$, some $r$. It
suffices to show that
\begin{equation}\label{ratiobound2} \frac{|(L^{M,N}(\vec{W}^L-\vec{w}^L))_i (x_j,t_k)|}{|(L^{M,N}\vec{\Phi})_i
(x_j,t_k)|} \leq C,
\end{equation}
for each $i=1,\;\dots \;,n,$ because an application of the
Comparison Principle then yields the required result.\\
The bounds on the denominator are given in \eqref{barrierbound4}.
To bound the numerator note that either $d_r>0$ or $d_r=0$.\\
Suppose first that $d_r>0$. Then the cases $i > r$ and $i \leq r$
are treated separately.\\
If  $i > r$, then, by \eqref{gen1} in Lemma \ref{general1},
\begin{equation}
|(L^{M,N}(\vec{W}^L-\vec{w}^L))_i(x_j,t_k)| \leq
C[M^{-1}+B^L_r(x_{j-1})+\frac{x_{j+1}-x_{j-1}}{\sqrt{\eps_{r+1}}}]
\end{equation}
Since $d_r>0$, by Lemma \ref{s1}, $B^L_r (x_{j-1}) = B^L_
r(\sigma_r -h_r) \leq CB^L_r(\sigma_r)= CN^{-2},$ and so
\begin{equation}
|(L^{M,N}(\vec{W}^L-\vec{w}^L))_i(x_j,t_k)| \leq
C[M^{-1}+N^{-2}+\frac{h_r+h_{r+1}}{\sqrt{\eps_{r+1}}}].
\end{equation}
Using \eqref{geom4} and the ordering of the $\eps_i$, these bounds
on the numerator and denominator lead
to \eqref{ratiobound2}.\\
If  $i \leq r$, then, by \eqref{gen2} in  Lemma \ref{general1},
\begin{equation}
|(L^{M,N}(\vec{W}^L-\vec{w}^L))_i(x_j,t_k)| \leq C[M^{-1}+\eps_i
\sum^r_{q=1}\frac{B^L_q(x_{j-1})}{\eps_q}+\frac{\eps_i}{\eps_{r+1}}\frac{x_{j+1}-x_{j-1}}{\sqrt{\eps_{r+1}}}].
\end{equation}
Since $d_r>0$, by Lemma \ref{s1}, $x_{j-1}=\sigma_r-h_r \ge
x^{(s)}_{q,r}$ for $1 \leq q \leq r-1$ and
\[\frac{B^L_q (x_{j-1})}{\eps_q} \leq \frac{B^L_r (x_{j-1})}{\eps_r} \leq \frac{B^L_r
(\sigma_r)}{\eps_r} = C\frac{N^{-2}}{\eps_r}.\] Thus
\begin{equation}
|(L^{M,N}(\vec{W}^L-\vec{w}^L))_i(x_j,t_k)| \leq
C[M^{-1}+\frac{\eps_i}{\eps_{r}}N^{-2}+\frac{\eps_i}{\eps_{r+1}}\frac{x_{j+1}-x_{j-1}}{\sqrt{\eps_{r+1}}}].
\end{equation}
Using \eqref{geom4} and the ordering of the $\eps_i$, these bounds
on the numerator and denominator lead
to \eqref{ratiobound2}.\\
Now suppose that $d_r=0$. Then $d_{r-1}>0$ and $D^+ h_r<0$,
because otherwise $x_j \notin J$. The cases $i \ge r$ and $i < r$
are now treated separately.\\
If  $i \ge r$, then, by \eqref{gen1} in Lemma \ref{general1} with
$r$ replaced by $r-1$,
\begin{equation}
|(L^{M,N}(\vec{W}^L-\vec{w}^L))_i(x_j,t_k)| \leq
C[M^{-1}+B^L_{r-1}(x_{j-1})+\frac{x_{j+1}-x_{j-1}}{\sqrt{\eps_{r}}}]
\end{equation}
Since $d_{r-1}>0$, by Lemma \ref{s1} with $r$ replaced by $r-1$,
\[B^L_{r-1} (x_{j-1})=B^L_ {r-1}(\sigma_r -h_r) \leq
CB^L_{r-1}(\sigma_{r-1})= CN^{-2},\] and so
\begin{equation}
|(L^{M,N}(\vec{W}^L-\vec{w}^L))_i(x_j,t_k)| \leq
C[M^{-1}+N^{-2}+\frac{h_r+h_{r+1}}{\sqrt{\eps_{r}}}].
\end{equation}
Using \eqref{geom4} and the ordering of the $\eps_i$, these bounds
on the numerator and denominator lead
to \eqref{ratiobound2}.\\
If  $i < r$, then by \eqref{gen2} in  Lemma \ref{general1} with
$r$ replaced by $r-1$,
\begin{equation}
|(L^{M,N}(\vec{W}^L-\vec{w}^L))_i(x_j,t_k)| \leq C[M^{-1}+\eps_i
\sum^{r-1}_{q=1}\frac{B^L_q(x_{j-1})}{\eps_q}+\frac{\eps_i}{\eps_{r}}\frac{x_{j+1}-x_{j-1}}{\sqrt{\eps_{r}}}].
\end{equation}
Since $d_{r-1}>0$, by Lemma \ref{s1}with $r$ replaced by $r-1$,
$x_{j-1}=\sigma_r-h_r \ge x^{(s)}_{q,r}$ for $1 \leq q \leq r-1$
and
\[\frac{B^L_q (x_{j-1})}{\eps_q} \leq \frac{B^L_{r-1} (x_{j-1})}{\eps_{r-1}} \leq \frac{B^L_{r-1}
(\sigma_{r-1})}{\eps_{r-1}} = C\frac{N^{-2}}{\eps_{r-1}}.\] Thus
\begin{equation}
|(L^{M,N}(\vec{W}^L-\vec{w}^L))_i(x_j,t_k)| \leq
C[M^{-1}+\frac{\eps_i}{\eps_{r-1}}N^{-2}+\frac{\eps_i}{\eps_{r}}\frac{x_{j+1}-x_{j-1}}{\sqrt{\eps_{r}}}].
\end{equation}
Using \eqref{geom4} and the ordering of the $\eps_i$, these bounds
on the numerator and denominator lead to \eqref{ratiobound2}. This
completes the proof.\eop \end{proof}

The following theorem gives the required first order in time and
essentially second order in space parameter-uniform error estimate.
\begin{theorem}Let $A(x,t)$ satisfy (\ref{a1}) and
(\ref{a2}). Let $\vec u$ denote the exact solution of (\ref{BVP})
and $\vec U$ the discrete solution of (\ref{discreteBVP}).  Then
\begin{equation}\;\; ||\vec{U}-\vec{u}|| \leq C\,N^{-2}(\ln N)^3. \end{equation}
\end{theorem}
\begin{proof}
An application of the triangle inequality and the results of
Theorems \ref{smootherrorthm} and \ref{singularerrorthm} leads
immediately to the required result.\eop
 \end{proof}

\end{document}